\documentclass[11pt]{article}
\usepackage{tipx}
\usepackage{amssymb,amsmath,amsthm, amsbsy,amsopn}
\usepackage{color}
\usepackage{upgreek}
\usepackage{yfonts}
\usepackage{mathrsfs}
\usepackage{stmaryrd}
\usepackage{hyperref}
\usepackage{showlabels}
\newtheorem{remark}{Remark}

\newtheorem{proposition}{Proposition}
\newtheorem{lemma}{Lemma}

\newtheorem{theorem}{Theorem}

\newcommand{\Z}{\mathbb{Z}}

\def\rest{\hskip 1pt{\hbox to 10.8pt{\hfill
\vrule height 7pt width 0.4pt depth 0pt\hbox{\vrule height 0.4pt
width 7.6pt depth 0pt}\hfill}}}

\def\beq{\begin{equation}}
\def\eeq{\end{equation}}

\def\B{{\mathbb B}}

\def\R{{\mathbb R}}
\def \ra {{\rm a_0 }}
\def \mM {{\mathcal M}}
\def \mT {{\mathcal T}}

\def \mQp {{\mathcal Q}_p}

\def\N{{\mathbb N}}
\def \rN{{\rm N}}
\def \Ext {\mathfrak E{\rm xt}}
\def \mpc{{\mathfrak p_{\rm c}}}
\def \mN {{\mathcal N}}
\def \mNcov {{\mathcal N}_{\rm cov}}

\def\S{{\mathbb S}}

\def \deg {{\rm deg \, } }

\def \rL{{\rm L}}
\def \rmc {{\rm m}_{\rm c}}
\def \rC {{\rm C}}

\def \rc {{\rm c}}
\def \rd {{\rm d}}
  \def \Liftp{ {{\mathcal L}^{\rm ift}_p}}

\def \fp {{\mathfrak p}}

\def \fM {{\mathfrak M}}

\def \trspace{{W^{1-{1\slash p}, p} (\partial \mM, \mN)}}
\def\Wunp {W^{1,p}}
\def\Wtrp {{W^{1-{1\slash p}, p}}}

\def \rE {{\rm E}}

\def \uobst {{\mathfrak u_{\rm obst}}}
\def \wobst {{\mathfrak w_{\rm obst}}}
\def \Text { {\mathcal  T}_{\rm ext}^p}
\def \TrL {{\rm T_{\rm {race, 0}}^{m, p}}}
\def \TrLq {{\rm T}_{{\rm race}, \mathfrak q_0}^{p}}
\def \TrLqmp {{\rm T}_{{\rm race}, \mathfrak q_0}^{m, p}}
\def \TrLqmpc {{\rm T}_{{\rm race}, \mathfrak q_0}^{\mm, \rmc}}
\def \Wp {{ {\rm W}_{m, \rm ct}^{1,p}}}

\def\sP {{\rm P}_{\rm south}}

\def \rk {{\rm k}}

\def \mm {{\mathfrak m}}
\def \mq {{\mathfrak q}}
\def \mi {{\mathfrak i}}
\def \mj {{\mathfrak j}}
\def \mr {{\mathfrak r}}

\def \mv {{\mathfrak v}}

\def \Cyl{  {\rm C}_{\rm yld} }
\def \mn {{\mathfrak n}}

\def \ms {{\mathfrak s}}
\def \Ext{ \mathfrak E^{\rm xt}}

\def\QED{\hbox{${\vcenter{\vbox{
   \hrule height 0.4pt\hbox{\vrule width 0.4pt height 6pt
   \kern5pt\vrule width 0.4pt}\hrule height 0.4pt}}}$}\vspace{7pt}}

\begin{document}
\author{Fabrice BETHUEL
\thanks{Sorbonne Universit\'es, UPMC Univ Paris 06, UMR 7598, Laboratoire Jacques-Louis Lions, F-75005, Paris, France } \
 \thanks{ CNRS, UMR 7598, Laboratoire Jacques-Louis Lions, F-75005, Paris, France} 
 }



\title{A  new obstruction to the extension problem for Sobolev maps  between manifolds}
\date{}
\maketitle

\begin{center}
{\it Dedicated to Ha\"im Brezis  on the occasion of his $70^{th}$ birthday.\\ His work and  friendship are a permanent\\ source of inspiration and motivation.}

\end{center}
\begin{abstract}
The main  result of the present paper, combined with   earlier results of  Hardt and Lin \cite {HL} settles the extension problem for $W^{1,p}(\mM, \mN)$, where $\mM$ and $\mN$  are compact riemannian manfolds, $\mM$ having non-empty smooth boundary and assuming moreover that $\mN$ is simply connected.  The main question which is studied is the following: Given a map in the trace space $\Wtrp (\partial \mM, \mN), $
 does it possess an extension in $\Wunp(\mM, \mN)$?   We show that the answer is  negative in the case $\mpc +1\leq p<m=\dim \mM$, where  the number $\mpc$ is related to the topology of $\mN$ and is defined in \eqref{deftrp}.  We also adress  the case $\mN$ is not simply connected, providing various results and rising some open questions. In particular,  we stress in that case  the relationship between the extension problem and the lifting problem to   the universal covering manifold. 
\end{abstract}

\section{Introduction}
\label {intro}
\subsection{The extension problem in the Sobolev class}
\label{mainresult}
 We consider in this paper two compact riemannian  manifolds  $\mM$ and $\mN$ with $\mN$ isometrically embedded in some euclidean space $\R^\ell$,  $\mM$ having a nonempty smooth boundary.  For  given $1< p <\infty$, we consider the Sobolev space  $W^{1, p} (\mM, \mN)$ of maps between $\mM$ and $\mN$ defined by
 \begin{equation*}
 W^{1, p} (\mM, \mN)=\{ u \in W^{1, p} (\mM, \R^\ell), \ u(x) \in N  {\rm \ for \ almost \ every \ }  x \in \mM \}. 
\end{equation*}
  By the trace theorem, the restriction of any map in $W^{1, p} (\mM, \mN)$ is a map in the trace space
   $W^{1- {1\slash p}, p} (\mM, \mN)$ defined by 
   \begin{equation}
   \label{palace}
   W^{1-{1\slash p}, p} (\partial \mM, \mN)=\{ u \in W^{1-1\slash p,  p} (\partial \mM, \R^\ell), \ u(x) \in N  {\rm \ for \ almost \ every \ }  x \in \mM \}, 
\end{equation}
   where the space $W^{1-1\slash p,  p} (\partial \mM, \R^\ell)$ is the  standard trace space  of maps from $\partial \mM$   to $\R^\ell$ for which the norm $ \Vert \cdot   \Vert_{1-1\slash p, p}$  is finite. The norm $\Vert u  \Vert_{1-1\slash p, p}$ is given by 
   \begin{equation*}
   \label{norm}
   \Vert u  \Vert_{1-1\slash p, p}=\Vert u \Vert_{L^p(\partial M)}+\lvert  u \rvert_{1-1\slash p, p} 
   \end{equation*}
    where the semi-norm ${\bf \lvert}  \cdot \rvert_{1-1\slash p, p} $ writes 
  \begin{equation}
  \label{seminorm}
\lvert u \rvert_{1-1\slash p, p}  = \left(\int_{\partial M} \int_{\partial M}  \frac {\vert u(x)-u(y) \vert^p}{\vert x-y \vert^{p+m-2 }} {\rm d}x {\rm d}y \right)^{\frac 1 p}.
   \end{equation}
   Given any map in $g$ in $W^{1-1\slash p,  p} (\partial \mM, \R^\ell)$, it is well-known that  there exists an extension $u$ of $g$ to the full domain $\mM$ such that $u \in W^{1, p} (\mM, \R^\ell)$ and $u=g$ on $\partial \mM$ in the sense of  the trace operator.  In the case we assume furthermore that the values of $g$ are  constrained to  belong to  $\mN$ so that  the map $g$ belongs to the space $W^{1-{1\slash p}, p} (\partial \mM, \mN)$,  a natural question, which has already been raised in several places  in the   litterature,  is to determine  wether we may find such an extension $u$ satistying moreover the  
   constraint on the target, that is $u(x) \in \mN$ for almost every $x  \in \mM$. Following  the notation introduced in \cite{BeD} we consider the subset $\mT^p(\partial \mM, \mN)$ of  $W^{1-{1\slash p}, p} (\partial \mM, \mN)$ defined  by 
   $$
 \mT^p_{\rm ext} (\partial \mM, \mN)  \equiv  \{u \in W^{1-{1\slash p}, p} (\partial \mM, \mN) {\rm  \  s.t \ }  \exists \,  U  \in 
    W^{1, p} (\mM, \mN) {\rm \ such \ that  \ }  U=u {\rm \ on \ } \partial \mM\}.  
    $$
    The extension problem for Sobolev mappings then   can be rephrased as: 
   
   \smallskip 
     \emph{$(\mathcal Q)_p$ Under which  conditions on $\mM$, $\mN$ and $p$ do we  have $\mT^p(\partial \mM, \mN)= W^{1-{1\slash p}, p} (\partial \mM, \mN)$?} 
    
   \medskip
    It  follows from Sobolev embedding that in the  case  $p > m=\dim \mM$   that maps in    $W^{1,p} (\mM, \mN)$ are in fact  continuous so that the answer  to question $\mQp$ completely  reduces to the corresponding  extension problem  for continuous maps between $\mM$ and $\mN$, a problem   in topology which might present  significant difficulties, depending on the nature of $\mM$.  The  same  answer holds for the  the\emph{ limiting  case} $p=\dim \mM$  (see Theorems 1 and 2 in \cite{BeD}).   We therefore restrict ourselves to the case $p<m$.  Since the nature of  our results is quite different in the two cases, we need to distinguish    the case when $\mN$ is simply connected  from  the case $\mN$ is not.

    \subsection{Statement of the result in the case $\mN$ is simply connected}
  We assume here that $\mN$ is simply connected, that is 
  \begin{equation}
  \label{simplet}
  \pi_1(\mN)=\{0\}.
  \end{equation}  
    It turns out in the case \eqref{simplet} holds, \emph{somewhat surprisingly}, that  question $(\mathcal Q)_{p}$ has a complete answer which depends only on $p$ and the topological  properties of the target manifold $\mN$.   In order to state our result,  we introduce the integer 
    \begin{equation}
    \label{deftrp}
   \mpc (\mN)=\inf\{\mj \in \N^*, \pi_{\mj}(\mN)\not = \{0\} \}. 
   \end{equation}
    For instance if the manifold $\mN$ is the $n$-dimensional sphere $\S^n$, with $n\geq 2$ so that \eqref{simplet} holds, then  $\mpc(\S^n)=n$. Notice that, since   $\mN$ is assumed to be compact $\mpc (\mN)<+ \infty$, and since it  is assumed to be  simply connected  $1<\mpc (\mN)$.   Our main result in the case \eqref{simplet} holds, which is actually also the main result of this paper,   can be stated as follows: 
     
     \begin{theorem} 
     \label{maintheo}
       Assume that $\mN$ is simply connected, i.e.  $\mpc (\mN)\not = 1$ and   let $1<p<m$. Then  we have
  \renewcommand{\theequation}{${\rm Ext}_p(\mM, \mN)$}
    \begin{equation}
         \label{theglaude}
    \mT^p_{\rm ext}(\partial \mM, \mN)= W^{1-{1\slash p}, p} (\partial \mM, \mN)
  \end{equation} 
      if and only if 
      \renewcommand{\theequation}{\arabic{equation}}
    \addtocounter{equation}{-1}
      \begin{equation}
      \label{thecondition}
      p<\mpc(\mN) +1.
      \end{equation}
 \end{theorem}
           
  We      recall  that the  fact that condition  $\eqref{thecondition}$   is \emph{sufficient} has already been proved  by Hardt and Lin  in \cite{HL}, where  a construction of an extension $U$ for any map $u \in  W^{1-{1\slash p}, p} (\partial \mM, \mN)$ is provided. The main result of this paper is hence the proof that the condition is also \emph{necessary}.  This  amounts, in the case $m>p> \mpc (\mN)+1$,   to construct a map in  $W^{1-{1\slash p}, p} (\partial \mM, \mN)$ which 
  cannot be extended as a $W^{1,p}(\mM, \mN)$ map.   Several earlier results have already pointed out  such obstructions in  various examples. 
   For instance,   it is shown in \cite{HL, BeD} that the existence of topological singularities  for maps in $\trspace$  in the case $\pi_{[p-1]}(\N)\not =\{0\}$ provides such obstructions to the extension.  In \cite{BeD}, the result of  Theorem \ref{maintheo} is proved in the case the target is the circle $\mN=\S^1$ (which is or course not simply connected\footnote{However  the proof of Theorem \ref{maintheo} provided in this paper carries over to this special case, see the discussion in subsection \ref{notsimply}}).
The obstruction there does not involve topological singularities and relies on \emph{lifting }  properties of $\S^1$-valued maps. 
 
       We  emphasize that the topology of $\mM$ does not  enter in the statement, in contrast with the case $p\geq m$ discussed before, for which the topology of $\mM$ might be an additional  source of obstructions. 
  As a matter of fact, the   core of our argument does not involve the topology of the domain and readily  deals with the case where  $\partial \mM \subset \R^{m-1}$, with a map which is constant off  the standard ball $\B^{m-1}$. More precisely, we prove:
  
  \begin{proposition}
  \label{mainprop}
  Assume that $\mpc (\mN)\not = 1$ and that   $\rmc\equiv \mpc(\mN) +1\leq p<m$.  There exists a map $\mathfrak u_{\rm obst}$  such that $\uobst=\mathfrak q_0$ on $\R^{m-1} \setminus \B^{m-1}$, where $\mq_0$ is an arbitrary point on $\mN$,   such that 
  \begin{equation}
  \uobst-\mq_0 \in W^{1-1\slash p, p} (\R^{m-1}, \R^\ell)  \  {\rm \ and \ } u(x) \in \mN {\rm \ for \ a. e.  \  } x \in \R^{m-1},  
  \end{equation}
    and such that there exist no map $U$  in  $W^{1,p} (\B^{m-1}\times [0,1], \mN)$  satisfying 
    $$U(\cdot, 0)=\uobst (\cdot)   {\rm \ on \ } \B^{m-1}\times \{0\}  {\rm \ in \ the \ sense \ of \ traces. }$$ 
\end{proposition}

   Theorem \ref{maintheo} is then deduced in a rather direct way from Proposition \ref{mainprop}. 
\subsection{The case $\mN$  is not simply connected}   
\label{notsimply}
We discuss in this paragraph the case when   $\mN$ is not simply connected, that is 
\begin{equation}
\label{notsimply}
\pi_1(\mN)\not = \{0\}.
\end{equation}
Several results of  topological flavor which enter in the proof of Proposition \ref{mainprop}   do not extend  to  the case $\mN$ is not simply connected,  this is in particular the case for  the  Hurewicz  isomorphism theorem, which is involved in some of our topological arguments.  It turns out that the case the manifold $\mN$ is simply connected is strongly related to properties and the nature  of the universal covering $\mN_{\rm cov}$  of $\mN$ as well as the  lifting property for Sobolev maps.   Let 
$$\Pi: \mNcov \to \mN$$
 denote  the covering map. If $\pi_1(\mN)=\{0\}$, then $\mNcov=\mN$ and $\Pi$ is the identity. The universal covering is always simply connected, that is $\pi_1(\mNcov) = \{0\}, $  so that 
$2 \leq \mpc(\mNcov) \leq +\infty$.  The simplest example is provided by the case    $\mN=\S^1$,  for which 
 $\pi_1(\S^1)=\Z$. In this example the universal covering is given by $\mNcov=\R$ and hence is not compact. The covering map is the exponential map  given by $\Pi(\theta)=\exp i \theta$ for $\theta \in \R$. Another classical example is given by the Lie group of rotations of the three-dimensional space $\mN=SO(3)$, for which 
 $\pi_1(SO(3))=\Z_2$. Here the covering space is  the group $SU(2)$, which,  in contrast to the first example,  is  compact.  As a matter of fact, an important observation is that $\mNcov$ is a \emph{compact Riemannian manifold} if  and only if  $\pi_1(\mN)$ is \emph{a finite group}.

\smallskip 
Given $p>1$, we say that a map $u \in \Wtrp (\partial \mM, \mN)$ is \emph{liftable} if and only if 
there exists a map $\varphi  \in \Wtrp (\partial \mM, \mNcov)$ such that 
\begin{equation}
\label{liftable}
u=\Pi \circ \varphi, 
\end{equation}
and that property $\Liftp (\partial \mM, \mN)$ holds if and only if every map $u \in \Wtrp (\partial \mM, \mN)$  is liftable.   A first elementary observation which stresses the close relationship    between  the lifting property $\Liftp (\partial \mM, \mN)$ and the extension problem is given  in the following result:

\begin{lemma}
\label{drouot}
Assume that $p\geq 2$ and that $\mM$  and $\partial \mM$ are simply connected.  If  the extension  property \ref{theglaude} holds, 
 then  the lifting  property $\Liftp (\partial \mM, \mN)$ holds also. 
 \end{lemma}
\begin{proof}
The proof relies on  the fact that the lifting property holds in the space $W^{1,p} (\mM,\mN)$ for $p\geq 2$,   that is given an arbitrary map $U \in W^{1,p} (\mM,\mN)$, there exists some $\Phi \in W^{1,p} (\mM,\mNcov)$ such that $U=\Pi \circ \Phi$ (see e.g. 
\cite{BeChi} Theorem 1 or \cite{PR}).  Since we assume  that the spaces  $\displaystyle{\mT_{\rm ext}^p(\partial \mM, \mN)}$ and $\displaystyle{W^{1-{1\slash p}, p} (\partial \mM, \mN)}$ coincide, it follows for  any map $u \in    \Wtrp (\mM, \mN)$  there exists a map $U$ in $W^{1,p} (\mM,\mN)$ such that $U=u$ on $\partial \mM$.
 Since $U=\pi \circ \Phi$, it follows in  view of the trace Theorem,  that 
$$ u=\pi \circ \varphi,$$
 where $\varphi$ is the trace of $\Phi$ on the boundary $\partial \mM$, so that $u$ possesses a lifting. The conclusion hence follows.  
\end{proof}

\smallskip
Lemma \ref{drouot}   shows  that  obstructions to the lifting property yield obstructions to the extension problem. The idea to use obstructions to liftings  was introduced    first in \cite {BeD}  to prove that in the special case $\mN=\S^1$,  the answer to the extension problem is negative for $ 3\leq  p<m$. The obstruction to the lifting property was then generalized in \cite{BBM} in the general setting of $W^{s, p} (\mN, \S^1)$  maps,   showing   that, turning back to our central problem\footnote{The  may also check that the construction introduced  in  the proof of Proposition \ref{mainprop} can be carried over to the special case $\mN=\S^1$, $2\leq p <m$, yielding hence an alternate proof}
\footnote{in the range $2\leq p <3$, the construction in \cite{HL} yields another obstruction},  obstructions  to extensions  appear for the exponents  $2\leq p<m$.   As matter of fact, this type of obstruction might be generalized to the case the 
$\mNcov$ is not compact, that is when $\pi_1(\mN)$ is infinite.  We have:

\begin{theorem}
\label{deux}  Assume that $\pi_1(\mN)$ is infinite.  Then  the extension property  \eqref{theglaude}  \emph{does not hold} for  $2\leq p<m$.  
\end{theorem}

In other words, the non-existence  part\footnote{which, as mentioned, is the main contribution of this paper}  of Theorem \ref{maintheo} remains valid in the case  $\mN$ is simply connected, provided the fundamental group is infinite.  Notice that  Theorem \ref{deux} does not cover the case $1\leq p <2$:  This leads to a first open question, namely prove (or disprove)  property \ref{theglaude} when 
 \begin{equation*}
\label{open1}
   \pi_1(\mN) {\rm \ is  \ infinite \  and\  non \ trivial \ } {\rm  \ and \ }  1\leq p <2\leq m.
   \leqno{({\mathcal O1})}
\end{equation*}
Let us actually mention that when $1\leq p <2\leq m$, then the lifting property $\Liftp(\partial \mM, \mN)$ holds (see Theorem 3 in  \cite{BeChi}). It would be tempting to conclude, in view of the construction of \cite{HL}   it that case,  that 
the answer is positive. However,  since $\mNcov$ is not compact, the  adaptation of the Hardt-Lin method   does not seem straightforward.  \\

 We finally turn to the case $\pi_1(\mN)$ is finite and non trivial. In this case also, we have only partial results.   We set 
$$ \tilde {\mathfrak p}_{\rm c}(\mN)= \mpc (\mN)=\inf\{\mj \in \N^* \setminus \{1\}, \pi_{\mj}(\mN)\not = \{0\} \}$$
 
 Since the homotopy groups of $\mNcov$ of order higher to 2 are equal to the homotopy groups of $\mN$ we actually have $\tilde {\mathfrak p}_{\rm c}(\mN)=\mpc(\mNcov)$.

\begin{theorem} 
\label{trois}
Assume that $\pi_1(\mN)$ is finite and non trivial.  

i)The extension property \ref{theglaude} \emph{does not hold} in the following two cases:
\begin{itemize} 
\item  $\tilde {\mathfrak p}_{\rm c}(\mN)+1 \leq p <m$.
\item
 $\displaystyle{2\leq p<3\leq m}$.
\end{itemize}

ii)  The extension property \eqref{theglaude} holds  if
$1\leq p <2$.

iii) If 
$3\leq p <\tilde {\mathfrak p}_{\rm c}(\mN)+1<m$, then  the extension property \eqref{theglaude} holds   \emph{if and only if}  the lifting property $\Liftp (\partial \mM, \mN)$ holds.
\end{theorem}

In view of the results described in Theorems \ref{maintheo}, \ref{deux} and \ref{trois}, the only other  case which remains open,  when $\mM$ is simply connected,  corresponds to   the case: 
\begin{equation*}
\label{open}
   \pi_1(\mN) {\rm \ is  \ finite \  and\  non \ trivial \ } {\rm  \ and \ }  3 \leq p <\tilde {\mathfrak p}_{\rm c}(\mN)+1<m.
   \leqno{({\mathcal O2})}
\end{equation*}
  Indeed, in this case it follows from Theorem \ref{trois} and Lemma \ref{drouot} that properties \ref{theglaude}  and 
$\Liftp (\partial \mM, \mN)$ are \emph{equivalent}. However, to  the author's knowlegde,   the later problem remains completely open  in the range of exponents $p$ considered. 

\begin{remark} {\rm  The lifting problem $\Liftp (\partial \mM, \mN)$ possesses  some  strong ressemblance  with the square (or  the  $k$-th) root  problem  for $\S^1$ valued maps in the Sobolev class.  This problem, which was addressed in \cite{BeChi}, is \emph{ solved } with a positive answer by Mironescu in \cite {miro, miro1} for $W^{1-\slash p, p}(\partial \mM, \S^1)$ maps, when  $p\geq3$. The proof relies on an ingenious  decomposition of the lifting, somewhat in the same spirit as the one introduced in \cite{BBM2}. 
These results \emph{might possibly}  suggest that the answer to $(\mathcal O)_2$  is also positive. 
}
\end{remark}

  Whereas  the proofs of Theorems \ref{deux} and \ref{trois} are essentially combinations of earlier known results (combined with Theorem
  \ref{maintheo}), the main contribution of the present paper is  Theorem \ref{maintheo} and its main ingredient Proposition \ref{mainprop}. The rest of this introduction  presents an outline of its proof. 
   \subsection{On the proof of Proposition \ref{mainprop}} 
     Let us  first show  that the map $\uobst$ constructed in Proposition \ref{mainprop} \emph{cannot be  not regular}.   In order to  get convinced of this fact,   we  introduce the set 
    \begin{equation}
    \label{tracer}
  \TrLq(\mN)=\TrLqmp (\mN)\equiv  \{ v \in \Wtrp(\R^{m-1}, \mN) {\rm \ with \ } v=\mq_0 {\rm \ on \ }   \R^{m-1} \setminus \B^{m-1} \} 
    \end{equation}
        and the quantity $\Ext_p(u)$ defined  for $u \in \TrLq(\mN)$ as 
     \begin{equation}
     \label{gammap}
     \Ext_p(u)=\Ext_{m, p}(u) =\inf \{\rE_p(U,\mathcal D_m),  U\in W_{\rm loc}^{1,p} (\mathcal D_m, \mN),\ U(x, 0)=u(x) {\rm \ for  \ } x \in  \R^{m-1} \}, 
   \end{equation}
   with the convention that the value is infinite when the defining set is empty,  where    the $p$-Dirichlet energy $\rE_p$ is  defined  for a domain $\Omega$ as
\begin{equation*}
\label{penergy}
 \rE_p(v, \Omega)=\int_\Omega \vert \nabla v \vert^p dx,  {\rm \ for \ } v: \Omega \to \R^\ell.
\end{equation*}
 and where we have also set
$$\mathcal D_m=\R^{m-1} \times [0,1].
$$ 
 In this setting,  Proposition \ref{mainprop} can be rephrased as 
      \begin{equation}
      \label{rephrased}
        \Ext_p (\uobst)=+ \infty.
        \end{equation}
 On the other hand, if $u$ belongs to the space $W^{1,p}(\B^{m-1}, \mN)$  then choosing as a comparison function in \eqref{gammap} the map $U$ defined on $\mathcal D_m$ by $U(x, t)=u(x)$ for $x \in \R^{m-1}$, and $t \in [0,1]$, then we are led to the inequality
      $$
      \Ext_p (u)\leq  E_p(u).
        $$
  Comparing this inequality with \eqref{rephrased}, we are led to the conclusion  that $\uobst$ \emph{does not belong} to   the space   $W^{1,p}(\B^{m-1}, \mN)$ and hence is not Lipschitz. However,  Although the map $\uobst$ is not regular,  an important intermediate step in the proof of Proposition \ref{mainprop}  is to obtain    lower bounds on $\Ext_p$ for specific lipschitz functions. We define for that purpose  for $u \in {\rm T}_{{\rm race}, \mathfrak q_0}^{m, \rmc}$ the quantity 
    \begin{equation}
    \label{cradoc}
   \mathcal I^{\rm xt}_{m}(u)=\inf \{\rE_{\mpc} \left(U,\Cyl^m\left (3 \slash 2\right) \right),  U\in \mathfrak W_m (u) \}, 
     \end{equation}  
    where we have set 
    \begin{equation}
    \label{define}
 \mathfrak W_m (u)=\{ U \in    W_{\rm loc}^{1,\rmc} (\Cyl^m\left (3 \slash 2\right), \mN),\ U(x, 0)=u(x) {\rm \ for  \ } x \in  \R^{m-1} \}
    \end{equation}
and, for $0\leq r \leq 2$, 
\begin{equation}
\label{chapeau}
  \Cyl^m(r)=\B^{m-1}(r) \times [0,\frac r 2].
  \end{equation}
Notice that, in  definition \eqref{cradoc}, we choose the exponent for the energy  functional $\rE_\mpc$ to be equal to $\mpc$,  whereas the integrability of the test  maps  $U$ is higher, since it is  assumed to equal to $\rmc=\mpc+1$.
   We have:
  
  \begin{proposition}
  \label{pirate}   Assume that  $\mpc (\mN)\not = 1$
  let $m$ be an integer such that $m \geq \mpc(\mN) $. For any integer $k \in \N^*$, there exists a  Lipschitz map $\mathfrak U_m^k$ in $\TrLqmp (\mN)$ such that 
  \begin{equation}
  \label{sept}
  \left\{
  \begin{aligned}
   & \Vert \nabla \mathfrak U_m^k \Vert_{L^\infty(\R^{m-1})} \leq  \rc_{m,1} k
  {\rm \ \, and \  \, } \\
   &\mathcal I^{\rm xt}_{m} (\mathfrak U_m^k) \geq \rc_{m,2} k^{\mpc} \geq \rc_3  \,  \rE_{\mpc}(\mathfrak U_m^k), 
   \end{aligned}
   \right.
   \end{equation}
    where $\rc_{m, 1}>0$, $\rc_{m, 2}>0$ and  $\rc_{m,3}$ are constants which do not depend on $k$.
  \end{proposition}
  
   We next describe some  observations  which lead to the proof of Proposition \ref{pirate}. 
  
 \subsubsection{ The linear extension operator}
    Consider first a map in $\Wtrp(\R^{m-1}, \R^\ell)$ such that $u=0$ on $\R^{m-1} \setminus \B^{m-1}$, then   the interpolation inequality for the $\Wtrp$ norm yields, for some universal constant $\rC_m$ depending only on $m$ 
    \begin{equation}
    \label{frodonet}
      \Vert u \Vert_{1-1\slash p, p} \leq \rC_m \Vert u \Vert_{1,p}^{1-1\slash p}   \Vert u \Vert_{p}^{1\slash p} \  {\rm \ provided \ }  u \in \Wunp(\R^{m-1}).
   \end{equation}
On the other hand, it follows from standard extension results that there exists a linear operator $\Text: \TrL \to  \Wp $, where 
\begin{equation*}
\left\{
\begin{aligned}
&\TrL=\TrL(\R^\ell)\equiv \{ v \in \Wtrp(\R^{m-1}, \R^\ell) {\rm \ with \ } v=0 {\rm \ on \ }   \R^{m-1} \setminus \B^{m-1} \}  \ {\rm \ and \ }  \\
&\Wp  \equiv  \{ V \in \Wunp (\R^{m-1}\times [0,1], \R^\ell){\rm \ with \ } v=0 {\rm \ on \ }  \left( \R^{m-1} \setminus \B^{m-1}(2)\right) \times [0,1]\}, 
\end{aligned}
\right.
\end{equation*}
 such that, if $U=\Text(u)$,  then 
\begin{equation}
\label{extension}
 \Vert U \Vert_{1,p} = \Vert \Text(u) \Vert_{1,p} \leq  C   \Vert u \Vert_{1-1\slash p, p}.
 \end{equation}
 
 
 Combining \eqref{extension} with  estimate \eqref{frodonet} we are led, in case $u \in \Wunp(\R^{m-1}, \R^\ell) $  to the estimate
\begin{equation}
\label{ineq}
\Vert \nabla  \Text(u) \Vert_{L^p(\mathcal D_m)} \leq \rC \Vert \nabla  u \Vert_{_{L^p(\R^{m-1})}}^{^{1-\frac 1 p}}  \Vert  u \Vert_{L^p(\R^{m-1})}^{\frac {1}{p}}. 
\end{equation}

We  turn back to $\mN$-valued maps. 
  Since the manifold $\mN$ is compact, we may choose some number $\rL$ such that $\vert y \vert \leq \rL$ for any $y \in \mN$, and hence  inequality \eqref{ineq} applies to any Lipschitz map $u \in  \TrLq (\mN)$ yields
  \begin{equation}
   \Vert u \Vert_{1-1\slash p, p} \leq  \rC_\rL \Vert \nabla  u \Vert_{_{L^\infty(\R^{^{m-1}})}}^{^{1-\frac 1 p}}.
  \end{equation}
  Setting $\Gamma_p^{\rm xt}(u))=\rE_p (\Text(u))$ we are led to the estimate
  \begin{equation}
  \label{douze}
  \Gamma_p^{\rm xt}(u) \leq   \rC \left( E_p(u)\right)^{1-\frac {1}{p}} \leq  \rC \Vert \nabla  u \Vert_{_{L^\infty(\R^{m-1})}}^{^{1-\frac 1 p}}.
 \end{equation}
  It is worthwhile to compare estimate \eqref{douze} with  the corresponding inequality \eqref{sept} for the  quantite $\Ext_p$ for the maps $\mathfrak U_m^k$ and to notice the  differences in the power laws  in term of the energy $\rE_p$  and the $L^\infty$ norm of the gradient as $k$ grows to $ + \infty.$
 
  \begin{remark}{\rm  In \cite {HL}, Hardt and Lin have succeded to show  that    the inequality
  \begin{equation}
  \label{hardtlin}
   \Ext_p(u) \leq   \rC \left( E_p(u)\right)^{1-\frac {1}{p}}
   \end{equation}
   holds  for $1\leq p< \mpc (\mN)+1$, 
    constructing a kind of    non linear analog of the operator $\Text$    which preserves the constraint on the target. Their proof   uses a tricky reprojection method.  Notice that in the special case $u$ is assumed to be moreover Lipschitz, then 
    \eqref{hardtlin} yields the estimate
   \begin{equation}
   \label{hardtlin2}
   \Ext_p(u) \leq   \rC \left \Vert \nabla u \right \Vert_{L^\infty( \R^{m-1})}^{p-1}. 
  \end{equation}
    }
     \end{remark}
 
  \begin{remark}
  \label{carraso}
  {\rm
   Proposition \ref{mainprop} shows that an inequality similar to \eqref{hardtlin2} does not hold for $ \mpc (\mN)+1<p<m$.
Indeed, we have 
\begin{equation}
\label{shakelton}
 \Ext_p(\mathfrak U_m^k ) \geq C_m k^p {\rm \ whereas  \ }     \Vert \nabla \mathfrak U_m^k \Vert_{L^\infty(\R^{m-1})} \leq  \rc_{m,1} k,
\end{equation}
the first inequality being a  consequence of  the second inequality in \eqref{sept} and inequality  \eqref{but} established in subsection \ref{introrem}. 
}
\end{remark}

 In the next paragraph, we will outline the main  topological nature of the obstruction to inequality \eqref{hardtlin} as well as the main ideas in the construction of Proposition \ref{mainprop}.
  \subsubsection{Conservation of topological fluxes}
We begin this subsection with a few elementary remarks of topological nature.  We start with a general observation concerning      the space $C_{\mq_0}^0 (\B^{m-1}, \R^\ell)$ defined by
 $$
 C_{\mq_0}^0 (\B^{m-1}, \R^\ell)=\{  w \in C^0(\B^{m-1}, \R^\ell),   \ {\rm s. t}  \  w(x)=\mq_0   {\rm \ on \ }
\partial \B^{m-1}
\  {\rm  \  for \ some \  } \mq_0\in \R^\ell \}. 
$$
Maps $v$ in $C_{\mq_0}^0 (\B^{m-1}, \R^\ell)$ will be considered sometimes as maps defined on the whole space $\R^{m-1}$ extending  their value by
$v(x)=\mq_0$ on $\R^{m-1}\setminus \B^{m-1}$, so that they are still continuous considered as maps on $\R^{m-1}$. We recall that $C_{\mq_0}^0 (\B^{m-1}, \R^\ell)$
may be mapped one to one to the space
  $C^0(\S^{m-1},  \R^\ell)$ thanks to the stereographic projection ${\rm St}_{m-1}$ which is  a smooth map from $\S^{m-1} \setminus \{\sP\}\subset \R^m $ onto   $\R^{m-1}$ and is   defined by 
 $$
 {\rm St}_{m-1} (x_1, \ldots  x_{m})= \left( \frac{x_1}{1+x_{m}}, \ldots,  \frac{x_{m-1}}{1+x_{m}} \right), 
 $$
 with $\sP=(0,0,0, \ldots, -1)$.  It follows that  given any map $v$  in $C^0_{\mq_0} (\R^{m-1}, \mN)$  the map $v \circ  {\rm St}_{m-1}^{-1}$ belongs to $C^0(\S^{m-1}, \mN)$. This allows to identify maps in   $C_{\mq_0}^0 (\B^{m-1}, \mN)$ with maps in $C^0(\S^{m-1}, \mN)$.  Moreover,  we have a one to one correspondance of homotopy classes.   Given a map $\varphi \in  C_{\mq_0}^0 (\B^{m-1}, \mN)$ we denote by $\llbracket \varphi \rrbracket$ its homotopy class. \\
 
 We consider next  a map $V\in C^0 (\R^{m-1} \times [0,1], \mN)$ and    the map $v$ defined on $\R^{m-1}$ by $v(x)=V(x, 0)$ for $x\in \R^{m-1}$.
  We  assume furthermore that 
  \begin{equation}
  \label{confitde}
   v =\mq_0  \ {\rm  \ on \ }  \R^{m-1}\setminus \B^{m-1}  {\rm \  so  \  that  \ }  
   v_{\vert_{\B^{m-1}}} \in C_{\mq_0}^0 (\B^{m-1}, \mN), 
   \end{equation}
   For $ 1 \leq r \leq 2$,  consider the cylinder  $\Cyl (r)=\Cyl^m(r)$, with $\Cyl^m$ defined in \eqref{chapeau} and 
  denote by $\Lambda^{m-1}(r)$  the inner part of the boundary defined by
  $$
  \Lambda^{m-1}(r)=\left(\partial \B^{m-1}(r) \times [0, \frac r 2] \right )\cup \B^{m-1} (r)\times \left\{\frac r 2 \right\} {\rm \ so \ that \ }  \partial \Cyl^m(r)=\Lambda_r^{m-1} \cup \B^{m-1}(r) \times \{0\}.
  $$
 Notice  that $\Lambda^{m-1}(r)$ may be mapped homeomorphically  to  the ball  $\B_r^{m-1}$ by a bilipschitz homeomorphism $\Phi_r$ whose Lipschitz constants may be bounded independently of $r$, that is $\Vert  \nabla  \Phi_r \Vert _\infty +\Vert  \nabla  \Phi_r^{-1}\Vert _\infty \leq C$.   Since,   in view of \eqref{confitde},  the restriction of the  map $V$ to $\partial \Lambda_r^{m-1}= \partial \B_r^{m-1}\times \{0\}$ is constant, we may define the homotopy class of its restriction to $\Lambda_r$ which we oriente according to the outer normal to $\partial \Cyl$. We  claim that, with this choice of orientation we have
 \begin{equation}
 \label{claim}
 \llbracket V_{\vert_{_{\Lambda^{m-1} (r)}}}\  \rrbracket=-\llbracket  v_{\vert_{\B^{m-1}}} \rrbracket  {\rm  \ for \ any \ } 1\leq r \leq 2. 
 \end{equation}
Indeed, since $V$ is continuous inside the cylinder $\Cyl(r)$, its  restriction to the boundary $\partial \Cyl (r)$, which is homeomorphic to the sphere $\S^{m-1}$,  has  trivial homotopy class. On the other hand we have
$$ \llbracket V_{\vert_{ \partial \Cyl(r)}}  \rrbracket =
 \llbracket V_{\vert_{_{\Lambda^{m-1} (r)}}}   \rrbracket + \llbracket v_{\vert_{\B^{m-1}}} \rrbracket$$
   so that  the conclusion \eqref{claim}  follows. 
The identity  \eqref{claim}  extends  to   Sobolev maps, provided the exponent $p$ is larger than $m$.  We have:

 \begin{lemma}
\label{selecta}
Assume that $p\geq m$ and  let  $v \in \TrLq(\mN)$ and $V\in W_{\rm loc}^{1,p}(\mathcal D_m, \mN)$ be such that
\begin{equation}
\label{meditation}
V(\cdot, 0)=v(\cdot) {\rm \ in \  the  \ sense \ of \ traces \ on \ } \R^{m-1}.
\end{equation}
Then, the  homotopy classes $\llbracket  v_{\vert_{\B^{m-1}}} \rrbracket $ and $\llbracket V_{\vert_{_{\Lambda^{m-1} (r)}}}\  \rrbracket$     are well defined  for every $1<r\leq 2$ and moreover \eqref{claim} holds. 
\end{lemma}
The proof is immediat for $p>m$, since it that case $V$ is continuous by Sobolev embedding. The limiting case $p=m$ requires more care and 
 follows  adapting ideas the from the works of Brezis and  Nirenberg \cite {BN1, BN2}.
 
 \begin{remark}{\rm  The result of Lemma \ref{selecta} \emph{does not hold} when $1\leq  p<m$, due to the possibility of having \emph{topological} singularities.  Assume  indeed that $\pi_{m-1} (\mN)\not = \{0\}$ and consider a map in  $v \in C_{\mq_0}^0 (\B^{m-1}, \R^\ell)$ having non trivial homotopy class and extended outside  $\B^{m-1}$ by $\mq_0. $
 Let $Q=(0,\ldots, 0, \frac 12 )\in \R^m$.  Given a point $M=(x_1, \ldots, x_{m-1}, x_m) \in \R^{m-1}\times [0,1]$   we set 
 \begin{equation}
 \label{retrodor}
 \left\{
 \begin{aligned}
 V(M)&=v(\Phi (M))  {\rm \ if \ } x_m > \frac 12 {\rm \ where \ } \Phi(M)= {\rm  D} (Q, M)  \cap  \R^{m_1}\times \{0\}  \\
  V(M)&=\mq_0 {\rm \ otherwise},  
\end{aligned}
 \right.
\end{equation}
 where $ {\rm D} (Q, M)$ denotes the line joining $Q$ to $M$. It follows that, if $p<m$, then  $v \in W^{1, p} (\mathcal D^m, \mN)$ with $V=v$ on $\B^{m-1} \times \{0\}$, the map $V$ being continuous except at the point $Q$, where it has a singularity  carrying a topological charge (the restriction to any small sphere around $Q$ has non trivial topology). On the other hand, we have, for $1\leq r \leq 2$
 $$
 V(M)=\mq_0 {\rm \ for \ }   M \in \Lambda_r^{m-1}  {\rm \ so \  that \ }  \llbracket V_{\vert_{_{\Lambda^{m-1} (r)}}}\  \rrbracket=\{0\}, 
 $$
  and hence $\llbracket V_{\vert_{_{\Lambda^{m-1} (r)}}}\  \rrbracket \not = -\llbracket  v_{\vert_{\B^{m-1}}} \rrbracket$.  Notice that the map $V$ no longer belongs to $W^{1,p}$ when $p\geq m$.
 }
 \end{remark}
 
 In the next section, we will see how these \emph{topological fluxes}  through the sets $\Lambda_r^{m-1}$ generate also \emph{energy fluxes}.
\subsubsection{ Infimum of energy  in homotopy classes and energy fluxes}
   For  an integer $ \mathfrak n \geq 2$ and an exponent $p\geq 1$ and a map $\varphi \in  {C}^0_{\mq_0} (\B^{\mn}, \mN)$,   we consider  the numbers 
$$
 \upnu_{_{\mn, p}}(\llbracket \varphi \rrbracket )=\inf \{ E_{p} (w),  w \in {\rm Lip} _{\mq_0}(\B^{\mn}, \mN) \ {\rm homotopic  \ to  \ } \varphi \}. 
$$
It follows from the scaling law for the energy
\begin{equation}
\label{fret}
\rE_p(u_r, \B^{\mn}(r))=r^{\mn-p} \rE_p(u, \B^{\mn}) {\rm \ where \ } u_r(x)=u (r x) \ {\rm \ for \ } x \in \B^{m},
 \rE_p (v_r)
\end{equation}
that,  for any  $1\leq p<\mn$, we have (letting $r$ go to zero in the above identity)
$$ \upnu_{_{\mn, p}}(\llbracket \varphi \rrbracket )=0  { \rm \ for  \  any \ homotopy \ class \ } \llbracket \varphi \rrbracket,  $$
whereas
 $$ {\rm when \ } p\geq \mn, {\rm \ then \ }  \upnu_{_{\mn, p}}(\llbracket \varphi \rrbracket )=0  {\rm \ if \ and \ only \  if \ } \llbracket \varphi \rrbracket=0.
 $$ 
 Going back to   Lemma \ref{selecta} and invoking scale invariance,  we obtain a lower bound for the energy on  surfaces $\Lambda(r)$, in the special case $m=\rmc$,  as stated in the next result.
   
   \begin{proposition} 
   \label{surfaces}
   Assume that  $p\geq m$ and  that $v$ and $V$ are as in Lemma \ref{selecta}.    Given $ p \geq s\geq m-1$,  we have,   for every $r \in [1, 2]$ and some constant $C_s>0$
\begin{equation}
\label{micromou}
   \int_{\Lambda^{m-1} (r)}\vert \nabla v \vert ^{s } \geq   C_s \upnu_{_{m-1, s}}(\llbracket v \rrbracket ).
   \end{equation}
\end{proposition}
 
 As a matter of fact, we will mainly  invoke this inequality with  the exponent $s=m-1$, so that we are led  to introduce the  numbers 
 \begin{equation}
 \label{matou}
 \upnu_{\mathfrak n}(\llbracket v \rrbracket )\equiv \upnu_{_{\mn, \mn}}(\llbracket v \rrbracket )  
 {\rm \ for \ } \mn \in \N^*.
 \end{equation}
   Combining H\"older's inequality with\eqref{micromou} we obtain for $V$ as in Lemma \ref{selecta} 
\begin{equation}
\label{micromou2}
   \int_{\Lambda^{m-1} (r)}\vert \nabla v \vert ^{s } \geq  C_s \left[\upnu_{_{m-1}}(\llbracket v \rrbracket )\right]^{\frac {s}{m-1}}  {\rm \ for \ } 1\leq r \leq 2 . 
   \end{equation}

We   discuss next some specific  properties of the numbers $\upnu_{_{m-1}}(\llbracket \varphi \rrbracket )$ in the special case
 \begin{equation}
 \label{dimcritic}
 m=\rmc\equiv  \mpc (\mN)+1.   
 \end{equation}
 when $\mpc\geq 2$.   In that case,  the manifold $\mN$ is   $(\mpc-1)$-connected  \footnote{recall that a manifold is said to be $\mathfrak q-1$ connected if $\pi_j(\N)=\{0\}$ for every integer $0\leq j <\mathfrak q. $}, a fact  which has important consequences on  the relevant  homotopy group $\pi_{\mpc} (\mN)$.  Such manifolds possess indeed  some  strong  similarities with  joints of spheres $\S^{\mathfrak q}.$ In particular, the homotopy group 
  $\pi_{\mpc} (\mN)$ is finitely generated and,  if  $\upsigma_1, \ldots , \upsigma_\ms$ denote the generators of  $\pi_{\mpc}(\mN)$, then the sub-groups generated by each of the $\upsigma_i's $ is infinite.  For $d\in \Z$ we set,  denoting by $\star$ the composition law in $\pi_{\mpc} (\mN), $
  $$\upsigma_i^d=\underset{ d  {\rm  \   times}}{\underbrace {\upsigma_i \star \ldots \star \upsigma_i}}. $$

   \begin{proposition}
   \label{clefkey}
   Assume that $\mpc\geq 2$.
     There exists constants $\rC_1>\rC_2>0$ depending only on $\mN$  such that   for any $i=1, \ldots,\ms $, we have
    \begin{equation}
    \label{troc}
 \rC_1 \vert d \vert  \geq    \upnu_{_{ \mpc}}(\upsigma_i^d) \geq \rC_2 \vert d \vert.  
     \end{equation}
   Moreover, for every $i=1, \ldots, \ms $ and every $d \in \Z$  there exists a Lipschitz map $\mv^i_d $ from $\B^{\mpc}$  to $\mN$ such that   
   $ \llbracket  \mv^i_d \rrbracket_{_i}=\upsigma_i^d$, 
   \begin{equation}
   \label{gluts}
   \vert \nabla \mv^i_d \vert^{\mpc} \leq \rc_0 \vert d \vert   {\rm \ in \ }   \B^\mpc {\rm \ and \ } 
  \mv^i_d=\mq_0  {\rm \ on \   } \partial \B^\mpc, 
 \end{equation}
     where  $\rc_0>0$ depends only on $\mN$ and where   $\mq_0\in \N$ is an arbitrary choosen point on $\mN$.
         \end{proposition}
  As a matter of fact, in the case $\mN=\S^p$, for which $\mpc=p$,  the  results in Proposition \ref{clefkey}    may be deduced directly from degree theory, whereas in the general case, we rely on some more sophisticated notions of topology, in particular related to the theory  of CW-complexes. 
 \subsubsection{On the construction of $\mathfrak U_m^k$ }
  We start describing  the construction  in  the case $m=\rmc=\mpc+1, $ which is actually the building block of the general case.   In that case, the construction follows directly from the construction in Proposition \ref{clefkey} since we set, for $k \in \N$
  \begin{equation}
  \label{prems}
  \mathfrak U_{\rmc}^k \equiv \mv^i_d {\rm \  with  \  } d=k^\mpc. 
   \end{equation}
 It turns out that, as a direct consequence of Proposition \ref{clefkey} and of Proposition \ref{surfaces},   that 
 the  map  $\mathfrak U_{\rmc}^k$ satisfies assumption \eqref{sept} for any $k \in \N^*$, provided  the constants $\rc_{\mpc,1}$ is choosen sufficiently large and the constant  $\rc_{\mpc, 2}$ are choosen sufficiently small,  a more  precise statement  being provided in Lemma \ref{bellaciao}.  The case  $m>\rmc=\mpc+1$ is deduced from the construction in  the critical dimension $m=\rmc=\mpc+1$ adding in a suitable way dimensions.
 \subsubsection{On the construction of $\uobst$}
   The map $\uobst$ is constructed gluing an \emph{infinite} but \emph{countable} number of scaled  and translated copies of the maps  $\mathfrak U_m^k$,  for suitable choices of  diverging indices  $k$ and shrinking  scaling factors.
The construction relies in an essential way on two properties. The first one is related to  the difference, for the maps   $\mathfrak U_m^k$, of   the asymptotic behaviors as $k$ grows  of the infimum of the energy of the extensions on one hand  and the $p$-th power of trace norm on the other. More precisely, we use extensively the fact that 
\begin{equation}
\label{cruxitude}
\Vert \mathfrak U_m^k -\mq_0 \Vert_{1-\frac 1p,  p }^p \leq C_m k^{p-1} {\rm \ whereas \ } 
  \Ext_{m, p} (  \mathfrak U_m^k) \geq \mathcal E^{\rm xt}_{m, p}( \mathfrak U_m^k)  \geq C_m k^p, 
\end{equation}
  where the quantity $\mathcal E^{\rm xt}_{m, p}(u)$, which is  localized version of  $\Ext_{m, p}$,  is defined in \eqref{credoc}.
  The second important property on which the construction is based upon is related  again on  scaling properties of the energy functional $\rE_p$. It may  an may be stated, for a general map $u: \R^m \to \mN$ and $0<r <2$ as the identity, similar to \eqref{fret},  namely
\begin{equation}
\label{scalingprop1}
 \rE_p(u_r, \Cyl^{m}(r))=r^{m-p} \rE_p(u, \Cyl^{m}) {\rm \ where \ } u_r(x)=u (r x) \ {\rm \ for \ } x \in \Cyl^m,
 \end{equation}
  so that in particular $\rE_p(u_r, \Cyl^{m}(r))$ tends to $0$ as  the scaling factor $r$ goes to zero. The scaling law 
  \eqref{scalingprop1} has a counterpart for the semi-norm $\lvert  \cdot \rvert_{1-1\slash p, p}$ given by the relation
  \begin{equation}
  \label{scalsemi}
  \lvert  u_r \rvert_{1-1\slash p, p}^p=r^{m-p} \lvert  u  \rvert_{1-1\slash p, p}^p  {\rm \ for \ } u: \R^{m-1} \to \R^\ell. 
  \end{equation}
   
   \medskip
 \noindent  
  {\it The gluing process.}  We define first  the set of  points $\{\fM_\mi\}_{\mi \in \N}$ in $\R^{m-1}$  where the copies of the maps $\mathfrak U_m^k$ will be glued by
\begin{equation}
\label{toujoursitude} 
\fM_\mi=  \left (   \underset {j=0} {\overset {\mi} \sum}   \updelta_\mj \right) \vec e_1 {\rm \ where \ }  \vec e_1=(1,\ldots ,0)\in \R^{m-1}, {\rm \ for \ } \mi \in \N, 
\end{equation}
  and  where we have set    
  $$
  \updelta_\mi=\frac{1}{\ra \mi  (\log \mi)^2}  {\rm \ for \ }  i \in \N^*, 
  {\rm \ with \ }
  \ra=2 \underset {j=0} {\overset {+ \infty} \sum}   \frac{1}{\mi  (\log \mi)^2}< + \infty.
  $$ 
 It follows that  the points $\fM_\mi$ are all on the segment joining the origin to the point 
 $$\fM_\star=\frac 12 \vec e_1= (\frac 12 , 0, \ldots, 0), $$ 
 converging  to the point $\fM_\star$ as $\mi \to + \infty$.
We  then  consider  a sequence of radii $ \displaystyle{(\mathfrak r_{\mathfrak i})_{\mi\in \N}}$   such that $ 0< \mr_\mi < \frac 1 4 \inf  \{\updelta_\mi, \updelta_{\mi-1}\}$ 
  and   the   corresponding collection of disjoint balls $ {(B_\mi)}_{\mi\in \N}$ given  by 
$$B_\mi \equiv \B^{m-1}(\fM_i, \mr_i)  {\rm \ for \   } \mi \in \N,  \ {\rm\  so \  that  \ }
 {\rm  dist } (B_\mi,  B_{\mathfrak j})\geq \frac 12  \sup\{ \updelta_i, \updelta_j\}  {\rm \ and \ }
 \underset {i\in \N} \cup B_\mi \subset \B^{m-1}. 
$$
    We  finally  introduce   a sequence of integers $\displaystyle{(\rk_\mi)_{\mi \in \N}}$ and   define  the map $\uobst $ on $\R^{m-1}$  as 
  \begin{equation}
  \label{mathcalitude}
\mathcal \uobst (x)=\mathfrak U_m^{\rk_\mi}\left (\frac {x-\fM_\mi}{\mr_\mi}\right) \ {\rm\  if \ } x \in B_\mi, \ \ 
  \mathcal U(x)=\mq_0  \  {\rm \  if \ } x \in \R^{m-1}\setminus \underset  {\mi \in \N}  \cup B_\mi.
  \end{equation} 
  The  next two results, which are directly connected to the scaling laws \eqref{scalingprop1} and  \eqref{scalsemi} reduce the constructionn of $:uobst$  to  the search of appropriate sequences  $ \displaystyle{(\mathfrak r_{\mathfrak i})_{\mi\in \N}}$ and  $\displaystyle{(\rk_\mi)_{\mi \in \N}}$. The first deals  with the trace semi-norm of $\uobst$.

  \begin{lemma}
  \label{thrace}
  Assume that 
 \begin{equation}
 \label{hypothesitude}
 \underset { \mi \in \N} \sum   \rk_\mi^{p-1}  \mr_i^{m-p}   <+\infty  {\rm \ and \  } \
 \mr_\mi  \leq  \frac {1}{16} \updelta_i.
 \end{equation}
 Then  the map $\uobst$ defined in \eqref{mathcalitude}  belongs to $\TrLqmp (\mN)$.
 
   \end{lemma}
   
    The second result  concernes the  energy of the extension. 
\begin{lemma}
\label{lemmitude}
 Assume that  $\rmc\equiv \mpc(\mN) +1\leq p<m$. Then we have
\begin{equation}
\label{labelitude}
\Ext_{m, p}(\uobst) \geq \underset { \mi \in \N}  \sum \rk_\mi^p \mr_i^{m-p}. 
\end{equation}
\end{lemma}
The proof of Proposition \ref{mainprop} is then completed by showing that there exists sequences  $ \displaystyle{(\mathfrak r_{\mathfrak i})_{\mi\in \N}}$ and  $\displaystyle{(\rk_\mi)_{\mi \in \N}}$ such that \eqref{hypothesitude} holds and such that 
\begin{equation}
\label{moyennitude}
\underset { \mi \in \N}  \sum \rk_\mi^p \mr_i^{m-p}=+\infty. 
\end{equation}
The fact that this is possible is related to the different exponents for $\rk_\mi$   ($p-1$ in the first one and $p$ in the second) in both inequality, a property which ultimately goes back to  \eqref{cruxitude}. 
\subsection{Outline of the paper}
\label{outline}
This paper is organized as follows. In the next Section we describe the relationship between energy  estimates and    topological invariants, in the case the exponent for the energy integral equals the dimension. In particular, we provide the proof to Proposition  \ref{clefkey}.  Section \ref{jack} is devoted to the the proof of  Proposition \ref{pirate}, whereas  the proof to Proposition \ref{mainprop} is given in Section \ref{sparrow}.  The proofs of the main theorems are finally completed in Sections \ref{black} and \ref{pearl}. 
 	
\numberwithin{theorem}{section} \numberwithin{lemma}{section}
\numberwithin{proposition}{section} \numberwithin{remark}{section}
\numberwithin{corollary}{section}
\numberwithin{equation}{section}

\section{Topology  and energy estimates }
\label{topenergy}
The main purpose of this section is to provide the proof of Proposition  \ref{clefkey}.  We split it into two parts, each of which corresponds to one of the two statements of the proposition, which  require however different assumptions.   the main focus is on the numbers $\upnu_\fp(\llbracket v \rrbracket)$ defined in   \eqref{matou}.  
We start the analysis with an explicit upper bound.  
\subsection{ An upper bound for  the energy in homotopy classes}

  Let 
$\mathfrak p\in \N^*$. We assume throughout this subsection that the $\mathfrak p$-th homotopy group of $\mN$ is non trivial that is $\pi_{\mathfrak p} (\mN)\not = \{0\}$  and that it is infinite. More precisely, we assume that there are elements $\upsigma_1, \ldots, \upsigma_{\mathfrak s}$ in $\pi_{\mathfrak p} (\mN)$  such that the sub-group 
$\mathfrak G_{i}$  generated by $\upsigma_i$  is infinite, that is 
\begin{equation}
\label{sim}
 \mathfrak G_i=\{ \upsigma_i^\ell, \ell \in \Z\}\sim \Z.
\end{equation}

\begin{lemma}
\label{upperb}
Assume that \eqref{sim}  holds.  There exists   a constant $c_1>0$ , such that  given any $i=1, \ldots, \ms$   and given any $d \in \Z$,   there exists  a map  $\Phi_d^i \in  
C^1_{\mq_0} (\B^\fp, \mN )$  such that $\llbracket  \Phi_d^i  \rrbracket= \upsigma_i^d$
and 
 \begin{equation}
 \label{labo}
\vert \nabla \Phi_d^i \vert (x)^\fp \leq \rc_0 \vert  d  \vert , {\rm \ for  \ any \  } x \in \B^\fp.  
\end{equation}
\end{lemma}

\begin{proof}  We  start with the  case $d=1$. Given $i=1, \ldots, \ms$ we choose an arbitrary map $\Phi^i= \Phi_1^i \in 
C^1_{\mq_0} (\B^\fp, \mN )$  such that $\llbracket  \Phi^i  \rrbracket= \upsigma_i$ and  set 
\begin{equation}
\label{delco}
 \rc_1=\Vert \nabla \Phi^i \Vert_{L^\infty (\B^\fp)}< +\infty. 
 \end{equation}
 It follows that \eqref{labo} is fullfilled in the case $d=1$, provided $\rc_0\geq \rc_1$. We next turn to the case  $d \geq 1. $ We introduce   the set of indices 
 $$A_\fp(d)=\{I=(i_1, i_2, \ldots i_\fp \}, i_k \in \N^*, (i_k)^\fp \leq d\}, $$ so that the total number of elements in $A_\fp$ is given by 
  $ \displaystyle{\sharp (A_\fp(d))= \left[ d^{\frac{1}{\fp} }\right]^\fp }$, where for $t\in \R^+$,  the symbol $[t]$ denotes  the largest integer less of equal to $t$.    Notice that,  by a convexity argument, we have
  $$d- \fp d^{1-\frac {1} {\fp}}  \leq \sharp (A_\fp(d))  \leq d {\rm \  \ so \  that  \  }  
0\leq r_d \equiv d-\sharp (A_\fp(d))  \leq   \fp d^{1-\frac {1} {\fp}} < \sharp (A_\fp(d)), 
  $$
 where the last inequality holds provided $d$ is sufficiently large.   We consider a subset $B_p(d)$ of $r_d$ distinct elements  in  $A_\fp(d)$.  We introduce the set of points $\Upsilon=\Upsilon_A \cup \Upsilon_B$, where $\Upsilon_A\equiv \{a_I\}_{I\in A_\fp(d)}$ and $\Upsilon_B= \{b_I\}_{I\in B_\fp(d)} $, the points $a_I$ and $b_I$ being defined, setting
  $h=d^{-\frac{1}{\fp}}$ by 
$$ a_I=\frac {h}{4} I  {\rm \ for \   } \ I \in A_\fp(d)  {\rm \ and \ } 
b_I=\frac {h}{4} I + (\frac 12, \ldots, 0)  {\rm \ for \   } \ I \in B_\fp(d),  
$$
so that the mutual distance between  distinct points in $\Upsilon$ is at least $\displaystyle{\frac h 4}$ and 
$ \sharp \Upsilon=d.$
We  then  define   the map $\Phi^i_d$ as
\begin{equation}
\label{defphi}
\left\{
\begin{aligned}
\Phi^i_d (x)&= \Phi^i ( \frac{ x-a_I}{ 8h})  {\rm \ for \ } x\in \B^\fp(a_I, \frac h 8), I \in A_\fp(d)  \\
\Phi^i_d (x)&= \Phi^i ( \frac{ x-b_I}{ 8h})  {\rm \ for \ } x\in \B^\fp(b_I, \frac h 8), I \in B_\fp(d)  \\
\Phi^i_d (x)&=\mq_0  {\rm \ otherwise.   \ } 
\end{aligned}
\right.
\end{equation}
Since $\Phi^i_d$ is obtained gluing $d$ scaled copies  copies of $\Phi^i$ its homotopy class is $\upsigma_i^d$, whereas combining \eqref{delco} with \eqref{defphi} we obtain  \eqref{labo} choosind $\rc_0=8  \rc_1$. This establishes the theorem for $d>0$. The proof is similar for $d<0$. 
\end{proof}

  Integrating the bound \eqref{labo} on  $\B^\fp$ and using the function $\Phi_d^i$ as a test function in the definition \eqref{matou} of  $\upnu_\fp(\upsigma_i^d)$ we are led to the upper bound
   \begin{equation}
   \label{gratis}
   \upnu_\fp(\upsigma_i^d)  \leq  C_2 \vert d  \vert, {\rm  \ for \ any \ } d \in \Z, 
   \end{equation}
where $C_2>0$ is some constant which does not depend on $d$.  This upper bound  actually corresponds to the right part of inequality \eqref{troc} and, 
 as seen above, this inequality does only require  the subgroup $\mathfrak G_i$ to be  infinite.   A natural question is  to determine whether  there exists also  in that case  a lower bound of the same magnitude, i.e. to know if  there exists a constant $C_i>0$ such that 
 \begin{equation}
 \label{cassegrain}
 \upnu_\fp(\upsigma_i^d)  \geq  C_i \vert d  \vert.
 \end{equation}
   Such a lower bound can be established   for instance if $\mN=\S^\fp$ using degree theory.  More precisely, in the case of the sphere $\S^\fp$, we have $\pi_\fp(\S^\fp)=\Z$, the unique generator of this  homotopy group being  the homotopy class of  the identity. In this case, the degree labels the order in the homotopy group. It is given by the integral formula
   \begin{equation}
   \label{formula}
    \deg u=\int_{\B^\fp} u^* (\omega) d\sigma. 
   \end{equation}
where $\omega$ is  a normalized volume form of the sphere and $^*$ denotes pull-back.    Formula \eqref{formula} yields rather directly to the upper bound \eqref{cassegrain}, in view of the pointwise inequality $\vert u^*(\omega)\vert \leq C \vert \nabla u \vert^\fp$. 
    It  turns out  however that the  bound \eqref{cassegrain} does not  hold for  general manifolds, even if \eqref{sim} holds.    This was proved for instance in \cite{Ri} for the case $\fp=3$ and $\mN=\S^2$ for which $\pi_3(\S^2)=\Z$. It is shown there that $\upnu_\fp(\upsigma^d)\leq  C \vert d \vert^{\frac 3 4}$, which contradicts \eqref{cassegrain} for large values of  $\vert d  \vert.$ 
 

\subsection{ A lower  bound for  the energy in homotopy classes}
\label{lowerenergy}
  In view of the previous remark and in order  to address  the bound \eqref{cassegrain}, we need to impose additional conditions on $\mN$.  In this subsection, we assume that $\fp \in \N^* \setminus \{1\} $ and impose that  the manifold $\mN$ is $(\fp-1)$-connected, that is we assume  throughout  that
\begin{equation}
\label{souple}
\pi_1(\mN)=\ldots =\pi_{\fp-1}(\mN)=\{0\}   {\rm \ and  \ }   \pi_\fp(\mN)\not =\{0\}.
\end{equation} 
This  kind of assumption is  for instance central in the statement of the Hurewicz isomorphism theorem and has also  been used in the context of Sobolev maps in several places in the literature  (see e.g. \cite{HL,Haj, P, PR} among others).  The main  feature which is used there is that $(\fp-1)$-connected manifolds possess strong analogies  with the sphere  $\S^\fp$, or  more precisely  with joints of $\fp$-dimensional spheres.  In particular the homotopy group has a finite number of generators $\upsigma_1, \ldots, \upsigma_{\mathfrak s}$  verifying  \eqref{sim}, corresponding to each of the spheres.  The lower bound for the $\fp$-energy  of 
$\S^\fp$-valued maps can be generalized to $(\fp-1)$-connected manifolds as follows: 
\begin{lemma}
\label{ouistiti}
Assume that \eqref{souple} holds. Then $\pi_\fp(\mN)$ is infinite. Moreover, if $\upsigma_1$ is a generator such that  \eqref{sim} holds, then  there exists a constant $C_i>0$ such that, for any $d \in \Z$, we have 
\begin{equation}
   \label{gratuitude}
   \upnu_\fp(\upsigma_i^d)  \geq  C_i \vert d  \vert. 
   \end{equation}
\end{lemma}

 Let us emphasize that this result is \emph{not new} and is actually  presumably well-known to the experts.   As a matter of fact, 
the result of Lemma \ref{ouistiti} can be directly deduced  as a special case  of Lemma 4.3 in \cite{PR}.  For sake of completeness however, we briefly explain the main ideas in the proof.

\begin{proof}[Sketch of the proof (following \cite{PR, Haj})]  
 The proof  relies on several observations, the   first ones being related  topological properties of the manifold $\mN$ we describe next.   

\bigskip
\noindent
{\it Topological background}.
We consider    some  smooth triangulation  $T$  of $\mN$ and denote by  $\mN^j$  the $j$-dimensional skeleton of $\mN$ for $1\leq j \leq \nu=\dim \mN$, so that 
$\mN^\nu=\mN$.  It turns out that, if $\mN$ is $(\fp-1)$- connected, then necessarily one has $\fp\leq \nu$ and   the 
 $\fp$-skeleton $\mN^\fp$ of $\mN$ has  the homotopy type of a joint of  $\mathfrak s$ spheres. Moreover,  the $\fp$- homotopy groups  of $\mN$ and $\mN^p$ coincide. We have therefore  
$$\displaystyle{ 
\mN^\fp\sim \underset{i=1} {\overset{\mathfrak s} \vee  } \S^\fp  {\rm \ and \ } \pi_\fp(\mN^\fp)= \pi_\fp(\mN).}
$$  
We denote by ${\tilde \upsigma}_i, \ldots, {\tilde \upsigma}_{\mathfrak s}$ the generators of $\mN^\fp$ which also correspond to generators of $\pi_\fp(\mN)$,  and set, for $\varphi \in C^0( \S^\fp, \mN^\fp)$
\begin{equation}
\label{decsigma}
 \langle  \varphi \rangle_{i, \fp}= d_i  {\rm \ if  \ } \ 
 \llbracket \varphi \rrbracket ={\tilde \upsigma}_1^{d_1}\star {\tilde \upsigma}_2^{d_2}\star\ldots\star 
 {\tilde \upsigma}_i^{d_i} \star  \ldots \star {\tilde \upsigma}_{\mathfrak s}^{d_{\mathfrak s}}.
 \end{equation}

 \medskip
 \noindent
 {\it Properties of maps in  $W^{1,\fp}(\S^\fp, \mN^\fp)$}. We restrict ourselves  for the moment to maps which take values on  the $\fp$-skeleton 
 $\mN^\fp \subset \mN$, and show that for such  a target the lower bound  \eqref{gratuitude} holds. 
Given $i=1, \ldots, \mathfrak s$, it can be proved that there exists a smooth \emph{"projection" map}  $\Pi_i: \mN^\fp \to \S^\fp$, with the property that, if $\varphi$ is a continous map from $\S^\fp$ to $\mN^\fp$, then $\Pi_i \circ \varphi \in C^0(\S^\fp, \S^\fp)$ with
\begin{equation}
\label{degradation}
{\rm  deg} (\Pi_i\circ \varphi)= \langle  \varphi \rangle_{i, \fp} {\rm \  for \ all \ }  \varphi \in C^0( \S^\fp, \mN^\fp).
\end{equation}
If $\varphi$ belongs moreover to  the space $W^{1,\fp} (\S^\fp, \mN^\fp)$,  then we have, since $\Pi_i$ is smooth, the pointwise inequality   $\vert \nabla (\Pi_i  \circ \varphi \vert \leq C \vert \nabla \varphi\vert$, so that 
\begin{equation}
\label{nounours}
\rE_\fp (\Pi_i \circ \varphi ) \leq C \rE_\fp  (\varphi). 
\end{equation}
On the other hand, since \eqref{gratuitude} holds for $\S^\fp$-valued maps thanks to degree theory,  we have, in view of \eqref{degradation}
$$
\rE_\fp  (\Pi_i \circ \varphi ) \geq  C \vert {\rm  deg} (\Pi_i\circ \varphi)  \vert
\geq C  \vert \langle  \varphi \rangle_{i, \fp}  \vert 
 $$
so that, combining  with \eqref{nounours}, we obtain, for  some constant  $C>0$, 
\begin{equation}
\label{pimprenelle}
\rE_\fp  (\varphi)  \geq   C \left \vert  \left  \langle  \varphi  \right \rangle_{i, \fp}  \right\vert
{\rm  \  \ for \ every \  } \varphi \in   W^{1,\fp}(\S^\fp, \mN^\fp).
\end{equation}

\medskip
 \noindent
 {\it   Projecting onto the $\fp$-skeleton  $\mN^k$}. This step  corresponds to an adaptation of reprojecton method introduced in \cite{HL},used for each of the individual  simplexes of the triangulation $T$.  This construction yields, for a  given   map $u \in {\rm Lip} (\S^\fp,\mN)$,  the existence of  another  map $\tilde u \in  {\rm Lip} (\S^\fp,\mN^\fp)$  such that, for some constant $C >0$ independent of $u$ 
\begin{equation}
\label{hhh}
\rE_\fp  (\tilde u)  \leq C \rE_\fp  (u )  {\rm \ and \  }  \
 \langle  \tilde  u \rangle_{i, \fp}= \langle    u \rangle_{i}  \ 
{\rm \   for  \ every \ } i=1, \ldots, \mathfrak s   
\end{equation}
and, moreover, if $u(x) \in \mN^\fp$  for some $x \in \S^\fp$, then we have $\tilde u(x)=u(x)$.   In \eqref{hhh}, we have set 
similar to \eqref{decsigma}
$$
 \langle    u \rangle_{i}=d_i   {\rm \ if  \ } \ 
 \llbracket u \rrbracket ={ \upsigma}_1^{d_1}\star { \upsigma}_2^{d_2}\star\ldots\star 
 {\upsigma}_i^{d_i} \star  \ldots \star { \upsigma}_{\mathfrak s}^{d_{\mathfrak s}}.
 $$
  Notice that  the construction of $\tilde u$ avec estimate \eqref{hhh} carries over to $W^{1,\fp}$ maps by a density argument.

 \medskip
 \noindent
 {\it  Proof of \eqref{gratuitude} completed}.   Consider some integer  $i \in \{1, \ldots, \mathfrak s\}$, some number $d \in  \Z$ and $u \in  W^{1,p} (\S^\fp,\mN^\fp)$ such that  $\langle    u \rangle_{i}=d$.   We claim that  there exists some constant $C>0$  which does not depend on $u$ nor on $d$ such that
 \begin{equation}
 \label{grouiner}
 \rE_p(u) \geq C \vert d\vert. 
 \end{equation}
 Indeed, in  view of the  results in previous paragraph, we may construct some map $\tilde u \in W^{1,p} (\S^\fp, \mN^\fp)$
such that 	$\rE_\fp  (\tilde u)  \leq C \rE_\fp  (u )$ and $ \langle  \tilde  u \rangle_{i, \fp}=d$.   Applying \eqref{pimprenelle} to $\tilde u$, we are led to  $\rE_\fp  (\tilde u)  \geq C \vert d \vert$. Combining the previous inequalities we  derive the proof of the claim \eqref{grouiner}. Finally, to establish \eqref{gratuitude}, it suffices to take the infimum in \eqref{grouiner} over all maps in the homotopy class. This completes the proof of Lemma \ref{ouistiti}.

\end{proof}

\subsection{Proof of Proposition \ref{clefkey}}	
\label{prkey}	
   We deduce from the definition of $\mpc$ that $\mN$ is $(\mpc-1)$-simply connected,  so that since \eqref{sim} holds
    for $\fp=\mpc$, we are in position to apply both Lemma \ref{upperb}  and Lemma  \ref{ouistiti}. Combining the  lower bound \eqref{gratis}  with the  upper bound  \eqref{gratuitude}, we derive  \eqref{troc}. Then,  choosing 
    $$ \mathfrak v_d^i= \Phi_d^i  {\rm \ for \ } d \in \Z, $$
     we observe that, thanks to \eqref{labo}, estimate   \eqref{gluts} is satisfied, which completes the proof.

\numberwithin{theorem}{section} \numberwithin{lemma}{section}
\numberwithin{proposition}{section} \numberwithin{remark}{section}
\numberwithin{corollary}{section}
\numberwithin{equation}{section}
\section{Proof of Proposition \ref{pirate} }
\label{jack}
\subsection{Introductory remarks}
\label{introrem}
We define  first  a few  quantities which  enter in the proof.
 For an integer $m \geq 1$,  an exponent $p>1$ and   given $u \in \TrLqmp(\mN)$ we introduce the  quantity 
   \begin{equation}
   \label{credoc}
    \mathcal E^{\rm xt}_{m, p}(u)=\inf \{\rE_{p} \left(U,\Cyl^m\left (3 \slash 2\right) \right),  U\in W_{\rm loc}^{1,p} (\Cyl^m\left (3 \slash 2\right), \mN),\ U(x, 0)=u(x) {\rm \ for  \ } x \in  \R^{m-1} \} .
   \end{equation}
      It follows from H\"older's inequality that for $p \geq \rmc$,
\begin{equation}
\label{wolferine}
 \mathcal I^{\rm xt}_{m}(u) \leq C_m \left( \mathcal E^{\rm xt}_{m, p}(u)\right)^{\frac {\mpc}{p}},
\end{equation}
where $\mathcal I^{xt}_m(u)$ is defined in \eqref{cradoc} and differs from  $\mathcal E^{\rm xt}_{m, p}$  by the choice of exponents both for the energy and the Sobolev maps, which are respectively $\mpc$ and $\rmc=\mpc+1$.
 On the other hand, it follows from the definition \eqref{credoc} that  we have the inequality
 \begin{equation}
 \label{classique}
     \mathcal E^{\rm xt}_{m, p}(u) \leq  \Ext_{m, p}(u),
 \end{equation}
  the main difference between these two quantities being that the domain of integration of the energy is smaller for the one on the left-hand side.  Combining \eqref{wolferine} with \eqref{classique}, we are led to  the lower bound for $\Ext_{m, p}(u)$ given by 
  \begin{equation}
  \label{but}
  \left( \mathcal I^{\rm xt}_{m}(u)\right)^{\frac{p} {\mpc}} \leq C_{m , p} \Ext_{m, p}(u),
  \end{equation}
  where $C_{m, p}>0$ denotes some constant depending only on $m$ and $p$. 
   The proof of Proposition \ref{mainprop} relies on a lower bound for   $\mathcal I^{\rm xt}_{m}(u)$ for  appropriate functions  $u$,  which immediately yields a lower bound for 
 $\Ext_{m, p}(u)$, in view of inequality \eqref{but}. The core  of the argument actually  deals with  the critical dimension  $m=\rmc$  with the choice of  the function $u=\mv^i_d.$  In several places, in particular when we increase dimensions,  we rely on the following lemma:
 
 \begin{lemma}
 \label{lefuneste}
  Let $f$ given an integrable non-negative  function  on  the cylinder
$\Cyl(R)$ for some $1\leq R \leq 2$. We have, 
\begin{equation}
\label{lefuneste}
\int_{\Cyl(R)} f (x) \rd  x \geq \frac 12\int_0^R  \left(\int_{\Lambda (r)} f (\sigma) \rd\sigma)\right ) \rd r.
\end{equation}
 \end{lemma} 
 \begin{proof}  Inequality \eqref{lefuneste} is  a consequence of the fact that the cylinder $\Cyl(R)$  may be decomposed as $\Cyl(R)=\underset{r\in [0, R]} \cup \Lambda(r)$ and of  Fubini's theorem (or  perhaps more precisely, the coarea formula).
 \end{proof}
\subsection{The critical dimension  $m=\rmc$}
 \begin{lemma}
 \label{bellaciao}
 We have, for some constant $\rc_0>0$ and any  number $d \in \Z$
 $$
  \mathcal I^{\rm xt}_{\rmc}(\mv^i_d) \geq \rc_0  \vert d \vert.
 $$ 
 \end{lemma}
\begin{proof}  We first notice that, since the function $v=\mv^i_d$ is Lipschitz, it belongs to the space $ \TrLq(\mN)$, for any $p\geq 1$.     Consider next an arbitrary map $V_d\in W_{\rm loc}^{1,\rmc} (\mathcal D_{\rmc}, \mN)$ such that 
$V_d(x, 0)=\mv^i_d(x)$ for $x \in  \R^{\rmc-1} $.     We  are  in position to apply   Proposition \ref{surfaces}   in dimension $m=\rmc$ to the functions $v=\mv^i_d$ and  $V_d$ with $p=\rmc$ and $s=\mpc=\rmc-1$.   It follows,  in view of \eqref{micromou} and the lower bound provided by \eqref{gluts},  that  for every $1\leq r  \leq  2 $  we have 
 \begin{equation}
\label{riritou}
\int_{\Lambda^{\rmc-1}(r)} \vert \nabla   V_d \vert ^{\rmc-1}\geq  C_{\rmc} \vert d \vert. 
\end{equation}
We apply the inequality  \eqref{lefuneste} to the function $f= \vert \nabla  V_d \vert^{\rmc-1}.$  This yields 
\begin{equation}
\label{lefuneste1}
\begin{aligned}
\int_{\Cyl(3\slash 2)}  \vert \nabla  V_d \vert^{\rmc-1} \rd  x &\geq 
\frac 12\int_{1}^{3\slash 2}  \left(\int_{\Lambda^{\rmc-1} (r)}  \vert \nabla  V_d \vert^{\rmc-1})\right ) \rd r \\
& \geq \frac 1 4 C_{\rmc} \vert d \vert, 
\end{aligned}
\end{equation}
 where,  for the inequality on the second line, we have invoked \eqref{riritou}.  On the other hand, we have, in view  of the definition of  $\mathcal I^{\rm xt}_{\rmc}(\mv^i_d)$
 $$
  \mathcal I^{\rm xt}_{\rmc}(\mv^i_d) =\inf\left \{ \int_{\Cyl(3\slash 2)}  \vert \nabla  V_d \vert^{\rmc-1} \rd  x,  
  V_d\in W_{\rm loc}^{1,\rmc} (\mathcal D_{\rmc}, \mN), {\rm s.t  \ } V_d(x, 0)=\mv^i_d(x)
  \right\},
$$
   so that the conclusion follows from \eqref{lefuneste1}.
\end{proof}
\subsection{Adding  dimensions}
\label{additude}
 Given an integer $\mm \in \N^*$, our first task will be  to construct\footnote{A similar construction is used in \cite{Be14}.} a mapping  
 $$
 \mathfrak I^{\mm}:  {{\rm T}_{{\rm race}, \mathfrak q_0}^{\mm, p}} (\R^\ell) \to \,  {{\rm T}_{{\rm race}, \mathfrak q_0}^{\mm+1,  p}} (\R^\ell), 
 $$
 which, to each    map $u: \R^{\mm-1}  \to  \R^\ell$  such that $u$ is constant equal to some value $\mq_0$ outside  the unit   ball $\B^{\mm-1}$, relates  a map 
 $ \mathfrak I^{\mm}(u): \R^{\mm} \to  \R^\ell $, constant equal to $\mq_0$ outside  the unit   ball $\B^{\mm}$. This map  is obtained by means of  a combination of several elementary geometric constructions, in particular a cylindrical rotation.  First, we consider the translated map $\tilde u$ defined on $\R^{\mm-1}$ by
 $$ \tilde u (x)=u(x-A^{\mm-1}) {\rm  \ where \ } A^{\mm-1} {\rm  \  denotes \ the \ point \ } \  A{\mm-1}=(2, 0,\ldots, 0) \in \R^{\mm-1}, 
 $$
  so that $\tilde u$ is equal to ${\mq}_0$  outside the ball $\B_1^{\mm-1}(A^{\mm-1})\subset \B_3^{\mm-1} (0)$.  We then  introduce the map 
  $T^{\mm}(u)$ defined  for   $(x_1, x_2, \ldots, x_{\mm-1}, x_{\mm})\in \R^{\mm}$ by 
  \begin{equation*}
T^{m} (u)(x_1, x_2,\ldots,  x_{\mm-1}, x_{\mm})= \tilde u(\mathfrak r (x_1, x_{2}), x_3, \ldots,x_{\mm-1},  x_{\mm}),  
\end{equation*}
where we have set $\mathfrak r(x_1, x_{2})=\sqrt{x_1^2+x_{2}^2}$.  It follows by construction that    the map $T^{\mm} (u)$ possesses cylindrical symmetry  around the  
$(\mm-2)$-dimensional hypersurface $x_1=x_2=0$. Moreover,  is equal to ${\mq}_0$  outside  the  ball $\B_3^\mm$ of radius $3$ and center the origin  and actually also on  the cylinder 
$\displaystyle{{[-\frac 12, \frac 12]}^2 \times \R^{\mm-2}}.$
 Since we wish the map $\mathfrak I^{\mm}(u)$ to be constant outside the unit ball  $\B_1^{\mm}$,   we need normalize  the previous map  and set 
\begin{equation}
\mathfrak I^{\mm}(u)(x)=T^{\mm}(u) (3x), {\rm \ for \ } x \in \R^\mm.
\end{equation}
It follows from the above observations that, as desired, the map  $\mathfrak I^{m}(u)$ equals ${\mq_0}$ outside  $\B_1^\mm$ and also on   the cylinder 
$$\mathcal Q^{\mm} \equiv { [-\frac 16, \frac 16]}^{2} \times \R^{\mm-2}.$$ 
The reader may easily prove the following:
\begin{lemma} 
\label{lip}
 The map $\mathfrak I^{\mm}$  is affine and  continuous 
 from $ {{\rm T}_{{\rm race}, \mathfrak q_0}^{\mm, p}} (\R^\ell)$  to $  {{\rm T}_{{\rm race}, \mathfrak q_0}^{\mm+1,  p}} (\R^\ell)$. If $u$ is a Lipschitz map in ${{\rm T}_{{\rm race}, \mathfrak q_0}^{\mm, p}} (\R^\ell)$, then $\mathfrak I^{\mm}(u)$ is also Lipschitz with
 $$ \Vert \nabla \mathfrak I^{\mm}(u)\Vert_{L^\infty (\R^{m})} \leq C_m  \Vert \nabla u \Vert_{L^\infty (\R^{m-1})}, 
 $$
 where $C_m>0$ denotes a constant depending only on $m$.
\end{lemma}

  We next specify somewhat the discussion to $\mN$-valued maps. We have:
  
  \begin{proposition}
  \label{rotextension} Assume that $m\geq \rmc$ and that    $u \in \TrLqmpc(\mN)$.  Then we have, for some constant $C_\mm>0$ depending only on $\mm$
  \begin{equation}
  \label{rotor}
 \mathcal I^{\rm xt}_{\mm+1}\left( \mathfrak I^{\mm}\left(u\right) \right)  \geq  C_\mm\mathcal I^{\rm xt}_{\mm} (u).
  \end{equation}
  \end{proposition} 
  \begin{proof}    The proof of \eqref{rotor} is actually mainly  a consequence of Fubini's theorem. In order to see this,  we introduce first some notation. 
 For $\theta \in \R$, we consider the  vector $\vec e_\theta= (\cos \theta, \sin \theta, 0, \ldots, 0)=\cos \theta\vec  e_1+ \sin \theta  \vec  e_{2}$ of $\R^{\mm}$ and set   $x_\theta= x. \vec e_\theta$, for $x \in \R^{\mm}$.    We introduce the  $(\mm-1)$-dimensional hyperplane $\mathcal P_\theta^{\mm-1}$ of $\R^{\mm}$ defined by
 $$
 \mathcal P_\theta^{\mm-1}\equiv {\rm Vect} \left \{\vec e_\theta, \, \vec e_3, \ldots,\,  \vec e_{\mm} \right\}
 $$
  and  the half-hyperplane $\mathcal P_\theta^{\mm-1, +}$ defined by
 \begin{equation}
 \label{peplum}
 \mathcal P_\theta^{\mm-1, +}=\{ x \in \mathcal P_\theta^{\mm-1},\  x_\theta\equiv x. \vec e_\theta \geq 0\}.
 \end{equation}
 We also consider  the  ball inside $\mathcal P_\theta^{\mm-1, +}$  centered at the point $\tilde A_{\theta}^{\mm-1,+}=\frac 2 3 (\cos \theta, \sin \theta, 0,  \ldots, 0)$  and of radius $R>0$ defined by 
$$
 \mathcal B_\theta^{\mm-1, +}(R)=\{x  \in  \mathcal P_\theta^{\mm-1, +},  0\leq  (x_\theta-\frac 2 3)^2+ x_3^2+ \ldots x_{\mm+1}^2 \leq R^2 \}. 
$$
 As above, if we consider a non-negative function $f$ defined on the domain  $\R^{\mm} \times [0,2]$ we have thanks to Fubini's Theorem  and for any $\displaystyle{0<R < \frac 2 3}$
 \begin{equation}
 \label{fou}
 \begin{aligned}
 \int_{\R^{\mm} \times [0,2]}  f(x)\ \rd r &=
 \int_0^{2\pi} \left(  \int_ {\mathcal P_\theta^{\mm-1, +}\times [0,2] } \vert x_\theta \vert f(x) \rd x \right) \rd \theta
\geq  \int_0^{2\pi} \left(  \int_ {\mathcal B_\theta^{\mm-1, +}(R)\times [0,2]} \vert x_\theta \vert f(x) \rd x \right) \rd \theta \\
& \geq  (\frac 2 3-R) 
   \int_0^{2\pi} \left(  \int_ {\mathcal B_\theta^{\mm-1, +}(R)\times [0,2] } f(x) \rd x \right) \rd \theta. \\
\end{aligned}
 \end{equation}
 Consider next an arbitray map $V \in \mathfrak W_{\mm+1} (\mathfrak  I^\mm (u))$, so that $V$ is defined on the $(\mm+1)$-dimensional cylinder $\Cyl^{\mm+1}(3 \slash 2)$ and  satisfies 
 \begin{equation}
 \begin{aligned}
  V(x_1, x_2, \ldots, x_{\mm}, 0)&= \mathfrak  I^\mm u (x_1, x_2, \ldots, x_{\mm})   
  {\rm \ for \  } x_1^2+ x_2^2+ x_{\mm}^2 \leq 1, \ x_1\geq 0 \\
  &=\tilde u (3 \mathfrak r (x_1, x_2),3 x_3, \ldots,3 x_{\mm}).
 \end{aligned}
 \end{equation}
    We apply  the identity \eqref{fou}  to the map 
 $$f={\bf 1}_{\Cyl^{\mm+1}(3 \slash 2)} \vert \nabla V \vert^{\rmc-1}  {\rm \ with \ radius \ } R=\frac 12.$$
This yields 
\begin{equation}
\label{drone}
\int_{\Cyl^{\mm+1}(3 \slash 2)} \vert \nabla V \vert^{\rmc-1} \geq   \frac {1}{6}   \int_0^{2\pi} \left(  \int_ {\mathcal B_\theta^{\mm-1, +}(1\slash 2)  \times [0,2] } \vert \nabla V \vert^{\rmc-1} \rd x \right) \rd \theta. \\
\end{equation}
  We claim that for any $\theta \in [0, 2 \pi]$, we have 
  \begin{equation}
  \label{claimitude}
  \int_ {\mathcal B_\theta^{\mm-1, +}(1\slash 2)  \times [0,1\slash 4] } \vert \nabla V \vert^{\rmc-1} \rd x  \geq C
  \mathcal I^{\rm xt}_{\mm} (u).
  \end{equation}
  \noindent
  {\it Proof of the claim \eqref{claimitude}}. 
Given an arbitray  map $v_\theta$ defined on  $\mathcal P_\theta^{\mm-1, +}\times [0,2] $, we define a map   $\mathfrak D^{\mm}_\theta(v_\theta)$ on the set $\R^{\mm-1} \times[0,2]$ setting  for $(x_1,  x_3,  \ldots, x_{\mm-1}, x_{\mm}, x_{\mm+1}) \in \R^{\mm-1} \times [0,2]$
\begin{equation}
\label{duxbellorum}
\mathfrak D^{\mm}_\theta(v_\theta)(x_1, x_3,  \ldots, x_{\mm-1}, x_{\mm}, x_{\mm+1}) =v_\theta(x_1\cos \theta , x_1 \sin\theta,  x_3, \ldots, x_{\mm-1}, x_{\mm}, x_{\mm+1}).
\end{equation}
It follows from this definition that the  energy $\rE_p$ is conserved in the sense that, for any $p\geq 1$ 
\begin{equation*}
\label{toute}
\int_{\B^{\mm-1}(\tilde A_0, 1\slash 2)\times [0,\frac 1 4]}  \vert \nabla \mathfrak D^{\mm}_\theta(v_\theta)\vert^p = \int_{\mathcal B_\theta^{\mm-1, +}(1\slash 2)  \times [0,1\slash 4] } \vert \nabla v_\theta\vert^p, {\rm \ where \ } \tilde A_0^{\mm-1}=(\frac 2 3, 0, \ldots, 0)\in \R^{\mm-1}. 
\end{equation*}
 We apply this construction to the restriction $V_\theta$ of the map $V$ to $\mathcal P_\theta^{\mm-1, +}\times [0,2]$.
  Since   the restriction of the  map $V$  on $\R^{\mm}\times \{0\}$ is equal to $\mathfrak I^\mm(u)$, we deduce  that 
\begin{equation}
\label{vicino}
\mathfrak D^{\mm}_\theta V_\theta (x',  0)=w(x')=u(3x'-A{m_1}) {\rm \  for \ any \ }  x'=(x_1, x_3, \ldots, x_{\mm} )     \in \R^{\mm-1}.
\end{equation}
We define  next  the map $\upzeta_\theta$ on  $ \Cyl^{\mm}(3 \slash 2)$ setting 
$$
\upzeta_\theta (x', s)= \mathfrak D^{\mm}_\theta V\left((\frac{x'}{3} +\tilde A_\theta), \frac s  3 \right)
{\rm \ for \ }   x'=(x_1, x_3, \ldots, x_{\mm} )     \in \R^{\mm-1}  {\rm   \	and  \ } s \geq 0.
$$
It follows from  \eqref{vicino} that 
$$
\upzeta_\theta (x', 0)=u(x')       {\rm \  for \ any \ }  x'=(x_1, x_3, \ldots, x_{\mm} )     \in \R^{\mm-1},      
$$
so that   the map $\upzeta_\theta$  belongs to  $\mathfrak W_{\mm} (u)$ and hence we have the inequality
\begin{equation}
\label{mya}
\int_{\Cyl^{\mm}(3 \slash 2)} \vert \nabla \zeta_\theta \vert ^{\rmc-1} \geq \mathcal  I_\mm^{\rm xt}(u). 
\end{equation}
On the other hand, we have 
\begin{equation}
\label{totoro}
\int_{\Cyl^{\mm}(3 \slash 2)} \vert \nabla \zeta_\theta \vert ^{\rmc-1}=\frac {1} {3^{\mm-\rmc+ 1}}
\int_{\B^{\mm-1}(\tilde A_0, 1\slash 2)\times [0,\frac 1 4]}  \vert \nabla \mathfrak D^{\mm}_\theta V \vert ^{\rmc-1}.
\end{equation}
Combining \eqref{toute}, \eqref{totoro} and \eqref{mya}  we complete the proof of the claim \eqref{claimitude}.

\smallskip
Going back to \eqref{drone}, we obtain, combining with \eqref{claimitude}
$$
\int_{\Cyl^{\mm+1}(3 \slash 2)} \vert \nabla V \vert^{\rmc-1} \geq   C \mathcal  I_\mm^{\rm xt}(u).
$$
Since this lower bound is true for any map $V$ in  $\mathfrak W_{\mm+1} (\mathfrak  I^\mm (u))$, it holds also for the infimum on that set yielding the desired conclusion \eqref{rotor}.  
\end{proof}
\subsection{Proof of Propostion \ref{pirate} completed}
  Recall that we have already defined the map $ \mathfrak U_{\mpc}^k$ in the critical dimension $m=\rmc$ by formula \eqref{prems}. We define the maps  $ \mathfrak U_m^k$ inductively on the dimension $m$  setting
\begin{equation}
\label{pour}
 \mathfrak U_{m+1}^k=\mathfrak I^m (\mathfrak U_m^k){ \rm \ for \ any \ } k \in \Z. 
\end{equation}
Combining the result of Lemma \ref{lip} with the properties of the map $\mathfrak v _d^i$  with $d=k^\mpc$ given in Proposition \ref{pirate}
 we obtain 
 \begin{equation}
 \label{ricca}
 \Vert \nabla \mathfrak U_{m}^k \Vert_{L^\infty(\R^{m-1})} \leq C_m  \Vert \nabla \mathfrak v _d^i \Vert_{L^\infty(\R^{m-1})}  \leq C_m k, 
 \end{equation}
  whereas Proposition \ref{rotextension} yields 
 \begin{equation}
 \label{labs}
 \mathcal I^{\rm xt}_{m}\left(\mathfrak U_{m}^k  \right)  \geq  C_m \mathcal I^{\rm xt}_{\mm} ( \mathfrak v _d^i ).
 \end{equation}
    The first inequality in \eqref{sept} is a direct consequence of \eqref{ricca}. For the second, we obtain,  combining  
   \eqref{labs} with the result of Lemma \ref{bellaciao} 
   $$
    \mathcal I^{\rm xt}_{m}\left(\mathfrak U_{m}^k  \right)  \geq  C_m \vert d \vert \geq  C_m  k^\mpc,
$$
 which yields the second inequality in \eqref{sept} and hence completes the proof. 
\section{Proof of Proposition \ref{mainprop}}
\label{sparrow}
In this section we will provide the proofs to Lemma \ref{thrace}  and  Lemma \ref{lemmitude} and then complete the proof 
of proposition \ref{mainprop}.

\subsection{On the trace norm of glued maps}
\label{radinitude}
 Whereas the energy norm  $W^{1,p}$ has a local nature,  the \emph{trace norm does not}. This introduces some interaction terms when computing the trace norm of glued maps. In order to estimate this interaction terms, we are led to consider the general situation where we are given a family of points $\{A_i\}_{i\in I}$ in $\R^{m-1}$,   a family of  radii $\{r_i\}_{i \in I}$  and maps in the  the subspace   $\mathfrak X_{m,p}(\{r_i, A_i\}_{i \in I})$  of $\in \Wtrp(\R^{m-1}, \R^\ell) $  defined by
 $$
 \mathfrak X_{m,p}(\{r_i, A_i\})
 =\left\{ u\in \Wtrp(\R^{m-1}, \R^\ell) {\rm \ such \ that \ } u=0 \ {\rm \  on \ }
 \R^{m-1} \setminus \underset{i\in I} \cup \B^{m-1}(r_i, A_i)
 \right\}.
 $$
 We assume furthermore that   the balls $\B^{m_1} (r_i, A_i)$ are well separated, that is we assume
 \begin{equation}
 \label{loin}
   \vert A_i-A_j\vert \geq 8 (r_i+r_j)  {\rm  \ for \ } i \not =j  {\rm \ in  \ } J.
  \end{equation}
Given a map  $u \in \Wtrp(\R^{m-1}, \R^\ell)$   we also introduce the  "localized  trace energy"
 $$
 \rN_{r,a} (u) =\int_{\B^{m-1}(2r, A)}\left( \int_{\B^{m-1}(2r, A)}  \frac {\vert u(x)-u(y) \vert^p}{\vert x-y \vert^{p+m-2 }} {\rm d}x \right){\rm d}y  {\rm  \  \ for \ }  a \in \R^{m-1} {\rm \ and \ } r>0, 
 $$
 When $u \in  \mathfrak X_{m,p}(\{r_i, A_i\}$ we will use the notation $\rN_{i} (u)= \rN_{r_i,A_i} (u)$.   The next result relates  the trace norm  to  the localized  trace energies.
 
 \begin{lemma}
 \label{trouville}
 Assume that \eqref{loin} holds and that $u \in \mathfrak X_{m,p}(\{r_i, A_i\})\cap L^\infty (\R^{m-1})$. Then we have, for some constant $C_m>0$ depending only on $m$
 \begin{equation}
 \label{localized}
  \vert u  \vert_{1-1\slash p, p}^p  \leq \underset {i \in I}  \sum \,   \rN_{i} (u)+  C_m \Vert u \Vert_\infty^p \underset {i \in I}  \sum  r_i^{m-p}.
 \end{equation} 
 \end{lemma}
 \begin{proof}  Set $\displaystyle{\Omega=\R^{m-1} \setminus \underset{i\in I} \cup \B^{m-1}(2r_i, A_i)}$ and  
 $\Omega_i=\R^{m-1} \setminus \B^{m-1}(2r_i, A_i)$ for $i\in I$. 
 We may decompose, in view of the defining formula \eqref{seminorm}   the quantity
 $ \vert u  \vert_{1-1\slash p, p}^p$ as
 \begin{equation}
 \label{decomposons}
 \vert u  \vert_{1-1\slash p, p}^p=\underset {i \in I}  \sum \,  \left( \rN_{i} (u)+ {\rm K}_i(u)\right) + {\rm R}(u), 
 \end{equation}
 where we have set 
 $$
{\rm K}_i(u) =\int_{\B^{m-1}(2r_i, A_i)}\left( \int_{\Omega_i}  \frac {\vert u(x)-u(y) \vert^p}{\vert x-y \vert^{p+m-2 }} {\rm d}x \right){\rm d}y 
 $$
  and
  $$
 {\rm R}(u)= \int_{\Omega }\left( \int_{\R^{m-1}}  \frac {\vert u(x)-u(y) \vert^p}{\vert x-y \vert^{p+m-2 }} {\rm d}x \right){\rm d}y. 
 $$
  Since $u(x)=0$ for $x \in \Omega_i$ and $u(y)=0$ for $y \in  \B^{m-1}(2r_i, A_i) \setminus \B^{m-1}(r_i, A_i)$, we deduce that in the integral defining  ${\rm K}_i(u)$, we have 
  $${\rm \  if \    } \vert u(x)-u(y) \vert \not =0, \ x \in \Omega_i  {\rm \ and \ }  y \in  \B^{m-1}(2r_i, A_i),  
  {\rm \ then  \ } \vert x-y \vert \geq  \vert x-A_i\vert-r_i.$$
   It follows that, invoking  also  the definition of $\Omega_i$  that
  \begin{equation}
  \label{grassouillet}
  \begin{aligned}
  {\rm K}_i(u) &\leq C\vert \B^{m-1}(2r_i, A_i) \vert \,  \Vert u \Vert_{_\infty}^p
   \int_{2r_i}^\infty \left(  \frac{1} {\varrho-r_i}\right)^{p+m-2 }\varrho^{m-2} d\varrho \\
  &\leq   C_m \, r_i^{m-p}  \Vert u \Vert_{_\infty}^p.
  \end{aligned}
  \end{equation}
 We argue somewhat similarly for ${\rm R} (u)$. Since $u(y)=0$ for $y \in \Omega$ it follows that 
\begin{equation}
\label{grasdouble}
\begin{aligned}
 {\rm R}(u)&= \int_{\Omega }\left( \int_{\R^{m-1}}  \frac {\vert u(x) \vert^p}{\vert x-y \vert^{p+m-2 }} {\rm d}x \right){\rm d}y 
 =\underset {i \in I} \sum  \int_{\Omega }\left( \int_{\B^{m-1}(r_i, A_i)}  \frac {\vert u(x) \vert^p}{\vert x-y \vert^{p+m-2 }} {\rm d}x \right){\rm d}y \\
 &\leq \Vert u \Vert_{_\infty}^p\underset {i \in I} \sum  \int_{\Omega }\left( \int_{\B^{m-1}(r_i, A_i)}  \frac {{\rm d}x}{\vert x-y \vert^{p+m-2 }}  \right){\rm d}y \\
 & \leq 
  \Vert u \Vert_{_\infty}^p\underset {i \in I} \sum   \int_{\B^{m-1}(r_i, A_i)} \left( \int_ {\Omega}  \frac {{\rm d}y}{\vert x-y \vert^{p+m-2 }}  \right){\rm d} x.\\
\end{aligned}
 \end{equation}
   Since ${\rm dist}( \B^{m-1}(r_i, A_i), \Omega ) \geq r_i$, we deduce,  that  for $x \in \B^{m-1}(r_i, A_i)$ we have 
   $$ 
    \int_ {\Omega}  \frac {{\rm d}y}{\vert x-y \vert^{p+m-2 }} 
    \leq C_m \int_{r_i}^{+\infty}  \frac{1} {\varrho^{p+m-2 }}\varrho^{m-2}d\varrho 
    \leq  C_m r_i^{-p+1}.
$$
 Going back to \eqref{grasdouble}  we are hence led to 
 \begin{equation}
 \label{dobile}
  {\rm R}(u) \leq C_m \Vert u \Vert_{_\infty}^p\underset {i \in I}   \sum  \vert  \B^{m-1}(r_i, A_i)\vert r_i^{-p+1}
  \leq C_m r_i^{m-p}.
\end{equation}
  Combining  \eqref{grassouillet}  and \eqref{dobile}  with \eqref{decomposons} we obtain the desired conclusion \eqref{localized}.
   \end{proof}

\subsection{Proof of Lemma \ref{thrace}}
\label{thracitude}
We apply  the result of Lemma \ref{trouville} to  the case $I=\N$, $r_\mi=\mr_\mi$ and $A_\mi=\fM_\mi$ for $i\in \N$, so that the map $\uobst-\mq_0$ constructed in \eqref{mathcalitude} with respect to the  given sequences $(\mr_\mi)_{\mi \in \N}$ and $(\fM_\mi)_{\mi \in \N}$  belongs to $\mathfrak X_{m,p}(\{\mr_\mi, \fM_\mi\}_{\mi \in I})$.  It also belongs to  $L^\infty(\R^{m-1})$ since $\uobst$ is $\mN$ valued. It follows from  the second assumption in   \eqref{hypothesitude} that  
 \eqref{loin} is satisfied. We are hence in position to apply  inequality  \eqref{localized} to $\uobst$. It  yields 
 \begin{equation}
 \label{loc}
  \Vert \uobst -\mq_0 \Vert_{1-1\slash p, p}^p  \leq \underset {\mi \in \N}  \sum \,   \rN_{\mi} (\uobst- \mq_0)+  C_m \rL^p \underset {i \in \N}  \sum  \mr_\mi^{m-p}.
 \end{equation} 
In view of the definition \eqref{mathcalitude} and  the scaling law \eqref{scalsemi}, we have
$$\rN_{\mi} (\uobst) \leq \mr_\mi^{m-p}\lvert \mathfrak U_m^k -\mq_0 \rvert_{1-\frac 1p,  p }^p
 \leq C_m \mr_\mi^{m-p} \rk_\mi^{p-1},
  $$
  so that going back to \eqref{loc} we obtain
  \begin{equation}
  \label{glouglou}
   \Vert \uobst -\mq_0 \Vert_{1-1\slash p, p}^p  \leq   C_m\underset {\mi \in \N}  \sum \, \mr_\mi^{m-p} (\rk_\mi^{p-1}+1).
  \end{equation}
   Since the right hand side of this inequality is finite in view of assumption \eqref{hypothesitude}, the conclusion follows.
   \qed
\subsection{Proof of Lemma \ref{lemmitude}}
We may assume that the set 
$$\mathcal Z_{m, p}=\{U\in W_{\rm loc}^{1,p} (\mathcal D_m, \mN),\ U(x, 0)=\uobst(x) {\rm \ for  \ } x \in  \R^{m-1} \}$$
 is not empty since otherwise  $\Ext_{m, p}(u)=+\infty$ and the proof is complete in that case. Let 
 $U\in  \mathcal Z_{m, p}$.  As a consequence   of the definition \eqref{mathcalitude}  and  the scaling law \eqref{scalsemi} 
 \begin{equation} 
 \begin{aligned}
 \rE_p \left(U, \Cyl(\frac 3 2\mr_\mi)+\fM_\mi \right) &\geq  \mr_\mi^{m-p} \mathcal E^{\rm xt}_{m, p}(\mathfrak U_{m}^{\rk_\mi}) \geq C_m \mr_\mi^{m-p}  \left(\mathcal I^{\rm xt}_{m}(\mathfrak U_{m}^{\rk_\mi})\right)^{\frac{p}{\mpc}}\\
& \geq   C_m \mr_\mi^{m-p}\rk_\mi^{p}.
\end{aligned}
 \end{equation}
   Since the collection of  sets  $(\Cyl(\frac 3 2\mr_\mi)+\fM_\mi)_{i\in \N} $  represents a collection of disjoints sets, we may sum up  the previous inequalities, which leads to the  inequality
   $$\rE_p(U,\mathcal D_m) \geq \underset {\mi \in \N} \sum C_m \mr_\mi^{m-p}\rk_\mi^{p}.
   $$
Taking the infinum over all maps in   $\mathcal Z_{m, p}$ we obtain the 
   desired conclusion.
 \qed
\subsection{Proof of Proposition \ref{mainprop} completed}
 We claim that  there exists a sequence of real positive numbers  $(\mr_\mi)_{\mi \in \N}$ and a sequence  of integers $(\rk_\mi)_{\mi \in \N}$ such that both \eqref{hypothesitude} and \eqref{labelitude} are satisfied. There is a large variety of possible  choices for such sequences, here we propose one of them. Setting for instance
 \begin{equation}
 \label{seteq}
 \mr_\mi= \left(  \frac{1}{\mi+1 }\right)^{\frac{p+1}{m-p}}  {\rm \ and \ } \rk_\mi=\mi+1.
 \end{equation}
 we verify that this choice satisfies assumptions  \eqref{hypothesitude} and \eqref{labelitude}.  With this choice of sequences, it  follows from Lemma \ref{thrace} that $\uobst$ belongs to $\Wtrp(\R^{m-1}, \mN)$, whereas Lemma \ref{lemmitude} shows that
 $$ \Ext_{m, p}(\uobst)=+\infty$$
  and hence has no finite energy extension, completing  the proof of Proposition \ref{mainprop}.
  \qed
\section{Proof of Theorem \ref{maintheo} }
\label{black}
We choose an arbitrary point of $A_0$ on $\mM$. Given $\varrho >0$, we consider the geodesic ball on $\mN$  centered at $A_0$ and of radius $\varrho$  given by
$$ 
B_{\rm geod} (\varrho,  A_0)=\{x \in \mM  {\rm   \ such \  that  \  } {\rm dist} _{\rm geod} (x, a_0)< \varrho\}, 
$$
where ${\rm dist} _{\rm geod}$ stands for the geodesic distance on $\mN$.  By standard results, there exists some 
$\varrho_0>0$  and a diffeomorphism $\Phi : \B_m^+ (2) \to B_{\rm geod} (\varrho_0,  A_0)$ such that 
$$\Phi \left(\B^{m-1}(2) \times \{0\}\right) =B_{\rm geod} (\varrho_0,  A_0) \cap \partial \mM, $$
 with 
 $$\B_m^+(2)=\{ x=(x', x_m) \in \B^m(2), {\rm with \ }  x'\in \R^{m-1}, x'\geq  0\}.$$
Assume next that $\mpc +1\leq p<m$. We define a map $\wobst \in \Wtrp(\partial \mM, \mN)$ setting
\begin{equation}
\left\{
\begin{aligned}
\wobst (x)&=\uobst( \Phi^{-1} (x))   \in x \in  B_{\rm geod} (\varrho,  A_0) \cap \partial \mM,  \\
\wobst (x)&=\mq_0  {\rm \ otherwise}.
 \end{aligned}
 \right.
 \end{equation}
  We claim that there exists no map $W \in W^{1,p}(\mM, \mN)$ such that $W(x)=\wobst (x)$ on $\partial \mM$.  Indeed, assume by contradiction that such a map $W$  does exist.   Let $\tilde W$  be the restrict of the map $W$ to the set 
  $B_{\rm geod} (\varrho_0,  A_0)$. Then the map $U=\Phi^{-1}  \circ \tilde W$ would belong to  
  $W^{1,p}(\B_m^+(2), \mN)$  with 
  $$\tilde W (x) = \uobst(x) {\rm \ for \ } x \in \B^{m-1}(2) \times \{0\}.
  $$
  This however contradicts the properties of $\uobst$ as stated in Proposition \ref{mainprop} and hence  shows that for  
  $\mpc +1\leq p<m$, the extension property does not hold. For the existence part, that is when  $1<p < \mpc+1$ we invoke the result in \cite{HL}, to assert that the existence property holds, so that the proof of Theorem \ref{maintheo}
  is complete. 
  \qed
\section{The case $\mN$ is not simply connected}
\label{pearl}
In this section, we provide the proofs of Theorem \ref{deux} and  Theorem \ref{trois}.
\subsection{Proof of Theorem \ref{deux}}
We assume that  $2\leq p<m$ and prove that,  if the assumptions of Theorem \ref{deux} are satisfied, then,  in that case,  the extension property \ref{theglaude} does not hold. This  is indeed a  consequence of Proposition 2 in \cite{BeChi}, which we briefly recall:  It  asserts that,  given $0<s<1$ and $p\geq 1$ such that $1\leq sp<m-1$ and assuming that    $\pi_1(\mN)$ is infinite, then there exists a map 
$u \in W^{s, p}(\partial \mM, \mN)$ such that $u$ can not be written as $u=\pi\circ \varphi$, with $\varphi  \in W^{s, p}(\partial \mM, \mNcov)$. We apply this result to the specific case which is of interest for us, namely the case $s=1-1\slash p$, so that $m-1> sp=p-1\geq 1$.   Proposition 2 in \cite{BeChi} hence shows that $\Liftp(\partial \mM, \mN)$ does not hold, and therefore nor does  the extension property \ref{theglaude}, in view of Lemma \ref{drouot}.  The proof is hence complete.  
\qed
\subsection{Proof of Theorem \ref{trois}}
\label{fuseaux}
 We  proceed distinguishing  four   cases. 
 
 \smallskip
 \noindent
{\it Case 1: $\tilde {\mathfrak p}_{\rm c} +1\leq p<m$}. It follows from Theorem \ref{maintheo} that 
${\rm Ext}_p(\mM, \mNcov)$ does not hold, hence there exists some map 
$\varphi \in W^{1-1\slash p,p} (\partial M, \mNcov)$ which cannot be extended as $W^{1,p}(\mM, \mNcov)$ map to the whole of $\mM$. Next we set
$$u= \Pi \circ \varphi, { \rm \ so \ that \ } u \in W^{1-1\slash p, p}(\partial \mM, \mN).
$$
We  claim that $u$ can not be extended as a $W^{1,p}(\mM, \mN)$ map to the whole of $\mM$.  To prove the claim, we assume by contradiction that there exists some map $U \in W^{1,p}(\mM, \mN)$ such that 
 $U(\cdot)=u(\cdot)$ on the boundary $\partial \mM$.   Since $p \geq 2$,  it follows from Theorem 1  in \cite{BeChi} that there exists some map $\Phi \in W^{1,p}(\mM, \mNcov)$ such that $U=\pi \circ \Phi$. Restricting this relation to the boundary, we are led to $\Phi (\cdot)=\varphi(\cdot)$ on $\partial \mM$, contradicting the fact that $\varphi$ cannot be extended and hence proving the claim. It follows that ${\rm Ext}_p(\mM, \mN)$ does not hold,  establishing the first assertion in  part i) Theorem \ref{trois}.

 \medskip
 \noindent
{\it Case 2 : $1\leq p<\tilde {\mathfrak p}_{\rm c} +1\leq m$ and property $\Liftp(\partial \mM, \mN)$ holds}.  We will show that in that case, given any map  $u \in W^{1-1\slash p, p}(\partial \mM, \mN)$ there exists a map $U \in W^{1,p}(\mM, \mN)$ such that  $U(\cdot)=u(\cdot)$ on the boundary $\partial \mM$.  Since we assume that $\Liftp(\partial \mM, \mN)$ holds, there exists some map $\varphi \in W^{1,p} (\partial M, \mNcov)$ such that $u=\Pi \circ \varphi$.    Applying Theorem \ref{maintheo} to the target $\mNcov$, we see that property ${\rm Ext}_p(\mM, \mNcov)$ holds, so that there exist a map $\Phi \in W^{1,p}(\mM, \mNcov)$  such that  $\Phi (\cdot)=\varphi(\cdot)$ on $\partial \mM$.   Setting $U=\pi \circ \Phi$, we obtain the desired map $U$. This proves that ${\rm Ext}_p(\mM, \mN)$ holds in the case considered. 
As a special case, we obtain the part iii)  of  Theorem \ref{trois}.

\medskip
 \noindent
{\it Case 3: $1 \leq p<2$.}  In this special case, it follows from Theorem 3 case ii) of \cite{BeChi} applied with 
$s=1-1\slash p$, so that $sp=p-1<1$, that $\Liftp(\partial \mM, \mN)$ holds. Hence the assumptions of  Case 2 are satisfied, so that we obtain that 
${\rm Ext}_p(\mM, \mN)$ holds in the case considered. 
 This yields the proof to  part ii) of   Theorem \ref{trois}.
 
 \medskip
 \noindent
{\it Case 4: $2\leq 3 \leq p<m$}.  In this case, since $\pi_{[p]-1}(\mN)=\pi_1(\mN)\not = \{ 0\}$, we obtain by \cite{BeD} (using the method of   \cite{HL}) topological obstructions to the extension problem, yielding  hence  the proof to the second statement in Theorem \ref{trois}, part i). 
 
 \medskip
 The three parts  of  Theorem  \ref{trois} are hence proved, so that the proof is complete.
 \qed


\end{document}